\documentclass[reqno,12pt]{amsart} 
\usepackage{amsmath}
\usepackage{amssymb}
\usepackage{amsthm}

\theoremstyle{plain} 
\newtheorem{theorem}{Theorem} 
\newtheorem{lemma}[theorem]{Lemma}

\theoremstyle{definition}
\newtheorem*{assumption (H1)}{Assumption (H1)}
\newtheorem*{assumption (H2)}{Assumption (H2)}

\theoremstyle{remark}
\newtheorem{remark}[theorem]{Remark}

\numberwithin{equation}{section}
\numberwithin{theorem}{section}

\newcommand\NoBlackBoxes{\global\overfullrule0pt}
\NoBlackBoxes
\theoremstyle{plain} 

\def\4{\kern1pt}

\def\6{\vphantom0}

\def\8{\kern-10pt}
\def\7#1{_{(#1)}}

\makeatletter
\let\serieslogo@\relax
\let\@setcopyright\relax

\def\speciallabelmark#1{\def\@currentlabel{#1}}
\makeatother

\begin{document}

\def\ffrac#1#2{\raise.5pt\hbox{\small$\4\displaystyle\frac{\,#1\,}{\,#2\,}\4$}}
\def\ovln#1{\,{\overline{\!#1}}}
\def\ve{\varepsilon}
\def\kar{\beta_r}

\title{STABILITY OF CRAMER'S CHARACTERIZATION \\
OF NORMAL LAWS IN INFORMATION DISTANCES}

\author{S. G. Bobkov$^{1,4,4'}$}
\thanks{1) School of Mathematics, University of Minnesota, USA;
Email: bobkov@math.umn.edu}
\address
{Sergey G. Bobkov \newline
School of Mathematics, University of Minnesota  \newline 
127 Vincent Hall, 206 Church St. S.E., Minneapolis, MN 55455 USA
\smallskip}
\email {bobkov@math.umn.edu} 

\author{G. P. Chistyakov$^{2,4}$}
\thanks{2) Faculty of Mathematics, University of Bielefeld, Germany;
Email: chistyak@math.uni-bielefeld.de}
\address
{Gennadiy P. Chistyakov\newline
Fakult\"at f\"ur Mathematik, Universit\"at Bielefeld\newline
Postfach 100131, 33501 Bielefeld, Germany}
\email {chistyak@math.uni-bielefeld.de}

\author{F. G\"otze$^{3,4}$}
\thanks{3) Faculty of Mathematics, University of Bielefeld, Germany;
Email: goetze@math.uni-bielefeld.de}
\thanks{4) Research partially supported by SFB 701}
\thanks{4') Research partially supported by the Humboldt Foundation}
\address
{Friedrich G\"otze\newline
Fakult\"at f\"ur Mathematik, Universit\"at Bielefeld\newline
Postfach 100131, 33501 Bielefeld, Germany}
\email {goetze@mathematik.uni-bielefeld.de}

\subjclass
{Primary 60E} 
\keywords{Cramer's theorem, characterization
of normal laws, stability problems} 

\begin{abstract}
Optimal stability estimates in the class of regularized distributions
are derived for the characterization of normal laws in Cramer's theorem 
with respect to relative entropy and Fisher information distance.
\end{abstract}

\maketitle
\markboth{S. G. Bobkov, G. P. Chistyakov and F. G\"otze}{Entropic 
stability in Cramer's theorem}

\def\theequation{\thesection.\arabic{equation}}
\def\E{{\bf E}}
\def\R{{\mathbb R}}
\def\C{{\bf C}}
\def\P{{\bf P}}
\def\H{{\rm H}}
\def\Im{{\rm Im}}
\def\Tr{{\rm Tr}}

\def\k{{\kappa}}
\def\M{{\cal M}}
\def\Var{{\rm Var}}
\def\Ent{{\rm Ent}}
\def\O{{\rm Osc}_\mu}

\def\ep{\varepsilon}
\def\phi{\varphi}
\def\F{{\cal F}}
\def\L{{\cal L}}

\def\be{\begin{equation}}
\def\en{\end{equation}}
\def\bee{\begin{eqnarray*}}
\def\ene{\end{eqnarray*}}

\section{{\bf Introduction}}
\vskip2mm
If the sum of two independent random variables has a nearly normal 
distribution, then both summands have to be nearly normal. This
property is called stability, and it depends on distances used
to measure ``nearness".
Quantitative forms of this important theorem by P. L\'evy are intensively 
studied in the literature, and we refer to [B-C-G3] for historical 
discussions and references. Most of the results in this direction
describe stability of Cramer's characterization of the normal laws 
for distances which are closely connected to weak convergence.
On the other hand, there is no stability for strong distances including 
the total variation and the relative entropy, even in the case where 
the summands are equally distributed.
(Thus, the answer to a conjecture from the 1960's by McKean [MC] is 
negative, cf. [B-C-G1-2].) Nevertheless, the stability with respect 
to the relative entropy can be established for {\it regularized} 
distributions in the model, where a small independent Gaussian noise 
is added to the summands. Partial results of this kind have been 
obtained in [B-C-G3], and in this note we introduce and develop 
new technical tools in order to reach optimal lower bounds 
for closeness to the class of the normal laws in the sense of 
relative entropy. Similar bounds are also obtained for the Fisher 
information distance.

First let us recall basic definitions and notations.
If a random variable $X$ with finite second moment has a density $p$, 
the entropic distance from the distribution $F$ of $X$ to the normal is 
defined to be
$$
D(X) = h(Z) - h(X) = \int_{-\infty}^\infty 
p(x) \log \frac{p(x)}{\varphi_{a,b}(x)}\,dx,
$$
where 
$$
\varphi_{a,b}(x) = \frac{1}{b\sqrt{2\pi}}\,e^{-(x-a)^2/2b^2},
\quad x\in \R,
$$ 
denotes the density of a Gaussian random variable $Z \sim N(a,b^2)$
with the same mean $a = \E X = \E Z$ and variance 
$b^2 = \Var(X) = \Var(Z)$ as for $X$ ($a \in \R$, $b>0$). Here
$$
h(X) = -\int_{-\infty}^\infty p(x) \log p(x)\,dx
$$
is the classical Shannon entropy, which is well-defined and 
is bounded from above by the entropy of $Z$, so that $D(X) \geq 0$.
The quantity $D(X)$ represents the Kullback-Leibler distance from $F$ 
to the family of all normal laws on the line; it is affine invariant, 
and so it does not depend on the mean and variance of $X$.

One of the fundamental properties of the functional $h$ is the entropy 
power inequality 
$$
N(X+Y) \geq N(X) + N(Y),
$$ 
which holds for independent random variables $X$ and $Y$, where 
$N(X)=e^{2h(X)}$ denotes the entropy power (cf. e.g. [D-C-T], [J]). 
In particular, if $\Var(X+Y)=1$, it yields an upper bound 
\begin{equation}\label{l1'}
D(X+Y)\le \Var(X)D(X)+\Var(Y)D(Y),
\end{equation}
which thus quantifies the closeness to the normal distribution for the sum in terms
of closeness to the normal distribution of the summands. The generalized Kac 
problem addresses (1.1) in the opposite direction: How can one bound 
the entropic distance $D(X+Y)$ from below in terms of $D(X)$ and $D(Y)$ 
for sufficiently smooth distributions? To this aim, for a small 
parameter $\sigma > 0$, we consider regularized random variables
$$
X_\sigma = X + \sigma Z, \qquad Y_\sigma = Y + \sigma Z',
$$
where $Z$, $Z'$ are independent standard normal random variables, 
independent of $X,Y$. The distributions of $X_{\sigma}$ and $Y_{\sigma}$ 
will be called {\it regularized} as well. 
Note that the additive white Gaussian noise is a basic statistical model used 
in information theory to mimic the effect of random processes that occur 
in nature. In particular, the class of regularized distributions contains 
a wide class of probability measures on the line 
which have important applications in statistical theory.

As a main goal, we prove the following reverse of the upper bound (\ref{l1'}).

\vskip5mm
\begin{theorem}
Let $X,Y$ be independent random variables with $\Var(X+Y) = 1$. Given 
$0 < \sigma \leq 1$, the regularized random variables $X_\sigma$ and 
$Y_{\sigma}$ satisfy
\begin{equation}\label{l1}
D(X_\sigma + Y_\sigma) \ge c_1(\sigma)
\left(e^{-c_2(\sigma)/D(X_\sigma)} + e^{-c_2(\sigma)/D(Y_\sigma)}\right),
\end{equation}
where $c_1(\sigma) = \exp\{c\sigma^{-6}\log \sigma\}$,
$c_2(\sigma) = c\sigma^{-6}$ with an absolute constant $c>0$. 
\end{theorem}

\vskip2mm
Thus, when $D(X_\sigma + Y_\sigma)$ is small, the entropic distances 
$D(X_\sigma)$ and $D(Y_\sigma)$ have to be small, as well. 
In particular, if $X+Y$ is normal, then both $X$ and $Y$ are normal,
so we recover Cramer's theorem. Moreover, the dependence with respect 
to the couple $(D(X_{\sigma}),D(Y_{\sigma}))$ on the right-hand side 
of (1.2) can be shown to be essentially optimal, as stated in 
Theorem 1.3 below. 

Theorem~1.1 remains valid even in extremal cases where $D(X)=D(Y)=\infty$
(for example, when both $X$ and $Y$ have discrete distributions). 
However, the value of $D(X_{\sigma})$
for the regularized variables $X_{\sigma}$ cannot be arbitrary.
However, for the regularized distributions the value of $D(X_{\sigma})$
cannot be arbitrary. Indeed, $X_\sigma$ has always a bounded density
$
p_\sigma(x) = \frac{1}{\sigma\sqrt{2\pi}}\ \E\,e^{-(x-X)^2/2\sigma^2}
\leq \frac{1}{\sigma\sqrt{2\pi}},
$
so that $h(X_\sigma) \geq -\log\frac{1}{\sigma\sqrt{2\pi}}$. This implies
an upper bound
$$
D(X_{\sigma}) \leq \frac{1}{2}\,\log\frac{e\Var(X_\sigma)}{\sigma^2} 
\leq \frac{1}{2}\,\log\frac{2e}{\sigma^2},
$$
describing a general possible degradation of the relative entropy for decreasing 
$\sigma$. If $D_\sigma \equiv D(X_{\sigma}+Y_{\sigma})$ is known to be 
sufficiently small, say, when $D_\sigma \leq c_1^2(\sigma)$, the inequality 
(\ref{l1}) provides an additional constraint in terms of $D_\sigma$, namely,
$$
D(X_{\sigma}) \leq \frac{c}{\sigma^6 \log(1/D_\sigma)}.
$$

Let us also note that one may reformulate (1.2) as an upper bound for 
the entropy power $N(X_\sigma + Y_\sigma)$ in terms of $N(X_\sigma)$ 
and $N(Y_\sigma)$. 
Such relations, especially those of the linear form
\begin{equation}\label{l2}
N(X+Y) \leq C\,(N(X) + N(Y)),
\end{equation}
are intensively studied in the literature for various classes of probability 
distributions under the name ``reverse entropy power inequalities", cf. e.g. 
\cite{C-Z}, \cite{B-M1}, \cite{B-M2}, \cite{B-N-T}.
However, (\ref{l2}) cannot be used as a quantitative version 
of Cram\'er's theorem, since it looses information about $D(X+Y)$, when
$D(X)$ and $D(Y)$ approach zero.

A result similar to Theorem 1.1 also holds for the Fisher information distance, 
which may be more naturally written in the standardized form
\begin{equation}\label{l3}
J_{st}(X) = b^2(I(X)-I(Z)) = b^2
\int_{-\infty}^\infty\Big(\frac{p'(x)}{p(x)}-\frac{\varphi'_{a,b}(x)}
{\varphi_{a,b}(x)}\Big)^2\,p(x)\,dx
\end{equation}
with parameters $a$ and $b$ as before. Here 
\begin{equation}\notag
I(X)=\int_{-\infty}^\infty \frac{p'(x)^2}{p(x)}\,dx,
\end{equation}
denotes the Fisher information of $X$, assuming that the density $p$ of 
$X$ is (locally) absolutely continuous and has a derivative $p'$ in the
sense of Radon-Nikodym. Similarly to $D$, the standardized Fisher 
information distance is an affine invariant functional, so that
$J_{st}(\alpha + \beta X) = J_{st}(X)$ for all $\alpha,\beta \in \R$, 
$\beta \ne 0$. In many applications it is used as a strong mesure of $X$ 
being non Gaussian. For example, $J_{st}(X)$ dominates 
the relative entropy; more precisely, we have
\begin{equation}\label{l4}
\frac{1}{2}\, J_{st}(X)\ge D(X).
\end{equation}
This relation may be regarded as an information theoretic variant 
of the logarithmic Sobolev inequality for the Gaussian measure due to 
Gross. Indeed it may be derived from an isoperimetric inequality 
for entropies due to Stam (cf. [S], [C], [B-G-R-S]). Moreover, 
in [S] Stam established an analog for the entropy power inequality, 
$\frac{1}{I(X+Y)} \geq \frac{1}{I(X)} + \frac{1}{I(Y)}$,
which implies the following counterpart of the inequality (\ref{l1'})
\begin{equation}
J_{st}(X+Y)\le\Var(X)J_{st}(X)+\Var(Y)J_{st}(Y),  
\label{l5}
\end{equation}
for any independent random variables $X$ and $Y$ with $\Var(X+Y)=1$.
We will show that this upper bound can be reversed in a full analogy 
with (\ref{l1}).

\vskip5mm
\begin{theorem}
Under the assumptions of Theorem~$1.1$,
\begin{equation}
J_{st}(X_{\sigma}+Y_{\sigma})
 \ge 
c_3(\sigma) \left(e^{-c_4(\sigma)/J_{st}(X_{\sigma})} +
e^{-c_4(\sigma)/J_{st}(Y_{\sigma})}\right),
\label{l6}
\end{equation}
where 
$c_3(\sigma) = \exp\{c\sigma^{-6}(\log\sigma)^3\}$,
$c_4(\sigma) = c\sigma^{-6}$ with an absolute constant $c>0$.
\end{theorem}

\vskip2mm
Let us also describe in which sense the lower bounds (\ref{l1}) and (\ref{l6})
may be viewed as optimal.

\vskip5mm
\begin{theorem}
For every $T \ge 1$, there exist independent identically distributed 
random variables $X=X_T$ and $Y=Y_T$ with mean zero and variance one, 
such that 
$J_{st}(X_{\sigma})\to 0$ as $T\to\infty$ for $0<\sigma\le 1$ and 
\begin{align}
D(X_{\sigma}-Y_{\sigma})
 &\le
e^{-c(\sigma)/D(X_{\sigma})} + e^{-c(\sigma)/D(Y_{\sigma})},\notag\\
J_{st}(X_{\sigma}-Y_{\sigma})
 &\le
e^{-c(\sigma)/J_{st}(X_{\sigma})} + e^{-c(\sigma)/J_{st}(Y_{\sigma})}\notag
\end{align}
with some $c(\sigma)>0$ depending on $\sigma$ only.
\end{theorem}

The paper is organized as follows. In Section~2 we describe preliminary 
steps by introducing truncated random variables $X^*$ and $Y^*$. Since their 
characteristic functions represent entire functions, this reduction of 
Theorems~1.1-1.2 to the case of truncated random variables allows to invoke 
powerful methods of complex analysis. In Section~3, $D(X_{\sigma})$ is 
estimated in terms of the entropic distance to the normal distribution 
for the regularized random variables $X_{\sigma}^*$, while an analogous 
result for the Fisher information distance is obtained in Section~4. 
In Section~5, the product of the 
characteristic functions of $X^*$ and $Y^*$ is shown  to be close to 
the normal characteristic function in a disk of large radius depending on 
$1/D(X_{\sigma}+Y_{\sigma})$ in the proof of Theorem~1.1 and on $1/J_{st}(X_{\sigma}+Y_{\sigma})$
in the proof of Theorem~1.2. 
In Section~6, we deduce  by means of saddle-point methods a special 
representation for the derivatives of the density of the random variables 
$X_{\sigma}$, which is needed in Sections~7-8.
Based on the resulting bounds for the density of $X_{\sigma}^*$, 
we establish the desired upper bounds for $D(X_{\sigma}^*)$ and 
$J_{st}(X_{\sigma}^*)$ in Sections~9 and 10, respectively.
In Section~11 we construct an example, showing the sharpness of the estimates 
of Theorems~1.1-1.2.

\section{{\bf Truncated random variables}}

Turning to the proof of Theorem 1.1, let us fix several standard notations. 
By
$$
(F * G)(x) = \int_{-\infty}^\infty F(x-y)\,dG(y), \qquad x \in \R,
$$ 
we denote the convolution of given distribution functions $F$ and $G$.
This operation will only be used when $G = \Phi_b$ is the normal 
distribution function with mean zero and a standard deviation $b>0$. 
We omit the index in case $b=1$, so that $\Phi_b(x) = \Phi(x/b)$
and $\varphi_b(x) = \frac{1}{b}\,\varphi(x/b)$.

The Kolmogorov (uniform) distance between $F$ and $G$ is denoted by
$$
\|F - G\| = \sup_{x \in \R} |F(x) - G(x)|,
$$
and $\|F - G\|_{\rm TV}$ denotes the total variation distance.
In general, $\|F - G\| \leq \frac{1}{2}\, \|F - G\|_{\rm TV}$, while 
the well-known Pinsker inequality provides an upper bound for the total 
variation in terms of the relative entropy. Namely,
$$
\|F - G\|_{\rm TV}^2 \leq
2 \int_{-\infty}^{\infty} p(x)\,\log \frac{p(x)}{q(x)}\,dx,
$$
where $F$ and $G$ are assumed to have densities $p$ and $q$, respectively.

In the required inequality (\ref{l1}) of Theorem 1.1, we may assume that
$X$ and $Y$ have mean zero, and that $D(X_{\sigma} + Y_{\sigma})$ is small. 
Thus, from now on our basic hypothesis may be stated as 
\be
D(X_{\sigma} + Y_{\sigma}) \le 2\varepsilon \qquad (0 < \ep \leq \varepsilon_0),
\en
where $\varepsilon_0$ is a sufficiently small absolute constant.
By Pinsker's inequality, this yields bounds for the total variation and 
Kolmogorov distances
\be
||F_{\sigma}*G_{\sigma}-\Phi_{\sqrt{1+2\sigma^2}}|| \le \frac{1}{2}\,
||F_{\sigma}*G_{\sigma}-\Phi_{\sqrt{1+2\sigma^2}}||_{\rm TV} \le 
\sqrt{\varepsilon} < 1, 
\en
where $F_{\sigma}$ and $G_{\sigma}$ are the distribution functions of 
$X_{\sigma}$ and $Y_{\sigma}$, respectively.
Moreover, without loss of generality, one may assume that
\begin{equation}\label{2.0}
\sigma^2\ge \tilde{c}(\log\log(1/\varepsilon)/\log(1/\varepsilon))^{1/3}
\end{equation}
with a sufficiently large absolute constant $\tilde{c}>0$.
Indeed if (\ref{2.0}) does not hold, the statement of the theorem
obviously holds.

In the inequality (\ref{l6}) of Theorem 1.2, we assume that $X$ and $Y$ 
have mean zero, and that $J_{st}(X_{\sigma} + Y_{\sigma})$ is small, i.e., 
\be
J_{st}(X_{\sigma} + Y_{\sigma}) \le 2\varepsilon \qquad 
(0 < \ep \leq \varepsilon_0),
\en
where $\varepsilon_0$ is a sufficiently small absolute constant. 
Then, according to the Stam inequality (\ref{l4}), the relative entropy 
has to be small as well, namely $D(X_{\sigma} + Y_{\sigma}) \le \varepsilon$.

Moreover, without loss of generality, we assume that
\begin{equation}\label{8.0}
\sigma^2\ge \hat{c}\log\log(1/\varepsilon)/(\log(1/\varepsilon))^{1/3}
\end{equation}
with a sufficiently large absolute constant $\hat{c}>0$.

We shall need some auxiliary assertions about truncated random variables. 
Let $F$ and $G$ be the distribution functions of independent, mean zero 
random variables $X$ and $Y$ with second moments 
$\E X^2 = v_1^2$, $\E Y^2 = v_2^2$, such that $\Var(X+Y)=1$. Put 
$$
N = N(\varepsilon) = \sqrt{1+2\sigma^2} \, 
\big(1+\sqrt{2\log(1/\varepsilon)}\,\big)
$$
with a fixed parameter $0<\sigma\le 1$.

Introduce truncated random variables at level $N$.
Put $X^*=X$ in case $|X|\le N$, $X^*=0$ in case $|X|>N$, and similarly 
$Y^*$ for $Y$. Note that
\begin{align}
&\E X^* \equiv a_1=\int_{-N}^N x\,dF(x), \qquad 
\Var(X^*) \equiv \sigma_1^2=\int_{-N}^N x^2\,dF(x)-a_1^2,\notag\\
&\E Y^* \equiv a_2=\int_{-N}^N x\,dG(x), \qquad 
\Var(Y^*) \equiv \sigma_2^2=\int_{-N}^N x^2\,dG(x)-a_2^2.\notag
\end{align}
By definition, $\sigma_1 \leq v_1$ and $\sigma_2 \leq v_2$. 
In particular,
$$
\sigma_1^2 + \sigma_2^2 \leq v_1^2 + v_2^2 = 1.
$$

Denote by $F^*,G^*$ the distribution functions of the truncated 
random variables $X^*,Y^*$, and respectively by $F^*_{\sigma},G^*_{\sigma}$
the distribution functions of the regularized random variables 
$X_{\sigma}^* = X^* + \sigma Z$ and $Y_{\sigma}^* = Y^* + \sigma Z'$, 
where $Z,Z'$ are independent standard normal random variables that are 
independent of $(X,Y)$.

\begin{lemma}\label{lem1}
With some absolute constant $C$ we have
\begin{equation}\notag
0\le 1-(\sigma_1^2+\sigma_2^2)\le CN^2\sqrt{\varepsilon}.
\end{equation} 
\end{lemma}

For the proof of Lemma~\ref{lem1}, we use arguments from [B-C-G3].
It will be convenient to divide the proof into several steps.

\vskip2mm
\begin{lemma}\label{lem2'}
For any $M>0$,
\begin{align}
1-F(M) + F(-M) 
 &\le 
2\,\big(1-F_{\sigma}(M) + F_{\sigma}(-M)\big)\notag\\
 &\le 
4\Phi_{\sqrt{1+2\sigma^2}}(-(M-2))+4\sqrt{\ep}. \notag
\end{align} 
The same inequalities hold true for $G$.
\end{lemma}

\vskip2mm
\begin{proof}
Since both $F_\sigma$ and $G_\sigma$ are continuous functions, 
one may assume that $\pm M$ are points of continuity for $F$ and $G$.
First note that 
\begin{align}
\frac{1}{2}\, F(-M) = \frac{1}{2} \int_{-\infty}^{-M}dF(u) & =
\int_{-\infty}^{-M} dF(u)\int_{-\infty}^0\varphi_{\sigma}(s)\,ds\notag\\
&\le\int_{-\infty}^{\infty}dF(u)\int_{-\infty}^{-M}\varphi_{\sigma}(s-u)\,ds=
F_{\sigma}(-M).\notag
\end{align}
In the same way, $\frac{1}{2}\, (1 - F(M)) \leq 1-F_{\sigma}(M)$, 
thus proving the first inequality of the lemma. For the second one, 
first note that $\E Y_\sigma^2 = v_2^2 + \sigma^2 \leq 2$, so that
$\P(|Y_\sigma| \geq 2) \leq \frac{1}{2}$, by Chebyshev's inequality.
Hence,
\bee
\frac{1}{2}\,F_\sigma(-M)
 & = & 
\frac{1}{2}\, \P(X_{\sigma}\leq-M) \\ 
 & \le &
\P(X_{\sigma}\leq-M, \, |Y_{\sigma}| \le 2)
 \ \le \ 
\P(X_{\sigma}+Y_{\sigma}\le -M+2) \\
 & = & 
(F_{\sigma}*G_{\sigma})(-M+2) \ \le \ 
\Phi_{\sqrt{1+2\sigma^2}}(-M+2) + \sqrt{\ep},\notag 
\ene
where we used (2.2) on the last step. By a similar argument,
$$
\frac{1}{2}\,(1-F_\sigma(M)) \leq 
\Phi_{\sqrt{1+2\sigma^2}}(-M+2) + \sqrt{\ep},
$$
thus proving the second inequality of the lemma for $F$. Similarly, one
obtains the assertion of the lemma for the distribution function $G$. 
\end{proof}

\begin{lemma}\label{lem2}
With some positive absolute constant $C$ we have
$$
||F^*-F||_{\rm TV} \le  C\sqrt{\ep}, \qquad 
||G^*-G||_{\rm TV}\le C\sqrt{\ep},
$$
$$
||F^*_{\sigma}*G^*_{\sigma}-\Phi_{\sqrt{1+2\sigma^2}}||_{\rm TV}\le 
C\sqrt{\ep}.
$$
\end{lemma}

\begin{proof}
The distribution $F^*$ of $X^*$ is supported on the interval $[-N,N]$,
where it coincides with $F$ as a measure up to an atom at zero of size 
$\P(|X|>N)$. Applying Lemma~\ref{lem2'} with $M = N$, we
therefore obtain that
\bee
||F^*-F||_{\rm TV} 
 & \le &
2\,\P(|X|>N) \ \leq \ 2\,\big(F(-N) + (1 - F(N))\big) \\
 & \leq &
8\,\Phi_{\sqrt{1+2\sigma^2}}(-N + 2) + 8\sqrt{\ep}.
\ene
By the choice of the function $N$, the latter expression does not exceed
$c\sqrt{\ep}$ up to some absolute constant $c>0$, so
$||F-F^*||_{\rm TV} \le c \sqrt{\varepsilon}$. Similarly, we have
$||G-G^*||_{\rm TV}\le c \sqrt{\varepsilon}$. From this,
using the triangle inequality, we conclude that 
$$
||F_{\sigma}*G_{\sigma} - F^*_{\sigma}*G^*_{\sigma}||_{\rm TV}
\le 2c \sqrt{\varepsilon}
$$ 
and then, by (2.2), 
$||F^*_{\sigma}*G^*_{\sigma} - \Phi_{\sqrt{1+2\sigma^2}}||_{\rm TV}
\le (2c+2)\sqrt{\varepsilon}$.
\end{proof}

\vskip2mm
{\bf Proof of Lemma~2.1}.
Since $\E\,(X^*_\sigma + Y^*_\sigma) = \E\,(X^* + Y^*) = a_1 + a_2$,
we have, integrating by parts,
\bee
a_1+a_2
 & = &
\int_{-\infty}^{\infty}x\,
d\big((F^*_{\sigma}*G^*_{\sigma})(x)-\Phi_{\sqrt{1+2\sigma^2}}(x)\big) \\
 & = & 
-\int_{-\infty}^{\infty}
\big((F^*_{\sigma}*G^*_{\sigma})(x)-\Phi_{\sqrt{1+2\sigma^2}}(x)\big)\,dx.
\ene
The modulus of the last integral does not exceed
\begin{equation}\notag
\int_{-4N}^{4N}
\big|(F^*_{\sigma}*G^*_{\sigma})(x)-\Phi_{\sqrt{1+2\sigma^2}}(x)\big|\,dx +
\int_{|x|>4N}
\big|(F^*_{\sigma}*G^*_{\sigma})(x)-\Phi_{\sqrt{1+2\sigma^2}}(x)\big|\,dx,
\end{equation}
where, by Lemma~\ref{lem2}, the first integral may be bounded by
$8NC\sqrt{\ep}$. Using the property that $(F^**G^*)(s) = 0$ for 
$s < -2N$ and $(F^**G^*)(s) = 1$ for $s > 2N$,
the second integral may be bounded by
\begin{align}
&\int_{|x|>4N}\,dx\Big(\int_{-2N}^{2N}
\varphi_{\sqrt 2\sigma}(x-s)|(F^**G^*)(s)-\Phi(s)|\,ds\notag\\
&+2\int_{|s|>2N}\varphi_{\sqrt 2\sigma}(x-s)
(1-\Phi(|s|))\,ds\Big)
\leq \int_{|x|>4N}\,dx\int_{-2N}^{2N}\varphi_{\sqrt 2\sigma}(x-s)\,ds\notag\\
&+2\int_{|s|>2N}(1-\Phi(|s|))\,ds 
\int_{|x|>4N}\varphi_{\sqrt 2\sigma}(x-s)\,dx\leq C\sqrt{\varepsilon}\notag
\end{align}
with some positive absolute constant $C$. Hence 
\begin{equation}\label{lem1.1}
|a_1+a_2|\le C_1N\sqrt{\varepsilon}
\end{equation}
with an absolute constant $C_1$.

An estimation of the second moment $\E\,(X^*_\sigma + Y^*_\sigma)^2$ is
based on the identity
\begin{align}
\int_{-\infty}^{\infty}x^2\,
d\big((F^*_{\sigma}*G^*_{\sigma})(x)-\Phi_{\sqrt{1+2\sigma^2}}(x)\big) =
-2\int_{-\infty}^{\infty} x
\big((F^*_{\sigma}*G^*_{\sigma})(x)-\Phi_{\sqrt{1+2\sigma^2}}(x)\big)\,dx.\notag 
\end{align}
Using the previous arguments, we obtain that
\begin{equation}\notag
\Big|\int_{-4N}^{4N}x\,
\big((F^*_{\sigma}*G^*_{\sigma})(x)-\Phi_{\sqrt{1+2\sigma^2}}(x)\big)\,dx\Big|
 \le 32N^2C\sqrt{\ep},
\end{equation}
while
\begin{align}
 &\Big|\int_{|x|>4N}x\,
\big((F^*_{\sigma}*G^*_{\sigma})(x)-\Phi_{\sqrt{1+2\sigma^2}}(x)\big)\,dx\Big|\notag\\
 &\le
\int_{|x|>4N}|x|\,dx\Big(\int_{-2N}^{2N}
\varphi_{\sqrt 2\sigma}(x-s)|(F^**G^*)(s)-\Phi(s)|\,ds\notag\\&+2\int_{|s|>2N}
\varphi_{\sqrt 2\sigma}(x-s)(1-\Phi(|s|))\,ds\Big)
\leq \int_{|x|>4N}|x|\,dx\int_{-2N}^{2N}\varphi_{\sqrt 2\sigma}(x-s)\,ds\notag\\
&+2\int_{|s|>2N}(1-\Phi(|s|))\,ds 
\int_{|x|>4N}|x|\,\varphi_{\sqrt 2\sigma}(x-s)\,dx
\leq  C_1\sqrt{\ep}.\notag
\end{align}
Hence, we finally get
\begin{equation}\label{lem1.2}
|\,\E (X^*+Y^*)^2-1| =
|\,\E (X^*_{\sigma}+Y^*_{\sigma})^2-(1+2\sigma^2)| \le C_2N^2\sqrt{\ep} 
\end{equation}
with some absolute constant $C_2$.
The assertion of the lemma follows immediately from (\ref{lem1.1}) and (\ref{lem1.2}). 
\qed

\vskip5mm
\noindent
{\bf Corollary 2.4.} {\it With some absolute constant $C$, we have
$$
\int_{|x|>N} x^2\,dF(x) \leq CN^2\sqrt{\ep}, \qquad
\int_{|x|>2N} x^2\,d(F_{\sigma}(x) + F_{\sigma}^*(x)) \leq CN^2\sqrt{\ep},
$$
and similarly for $G$ replacing $F$.
}

\vskip2mm
\noindent
{\it Proof.} By the definition of truncated random variables,
$$
v_1^2 = \sigma_1^2 + a_1^2 + \int_{|x|>N} x^2\,dF(x), \qquad
v_2^2 = \sigma_2^2 + a_2^2 + \int_{|x|>N} x^2\,dG(x),
$$
so that, by Lemma 2.1,
$$
\int_{|x|>N} x^2\,d(F(x) + G(x)) \leq 1 - (\sigma_1^2 + \sigma_2^2) \leq
CN^2\sqrt{\ep}.
$$
As for the second integral of the corollary, we have
\bee 
\int_{|x|>2N} x^2\,dF_{\sigma}(x)
 & = &
\int_{|x|>2N} x^2 
\bigg[\int_{-\infty}^{\infty}\varphi_{\sigma}(x-s)\,dF(s)\bigg]\,dx \\
 & = &
\int_{-\infty}^{\infty} dF(s)
\int_{|x|>2N}x^2\varphi_{\sigma}(x-s)\,dx \\
 &\le & 
2\int_{-N}^{N}s^2\,dF(s) \int_{|u|>N}\varphi_{\sigma}(u)\,du + 
2\int_{|s|>N}s^2\,dF(s)\int_{-\infty}^{\infty} \varphi_{\sigma}(u)\,du \\
 & & \hskip-5mm + \
2\int_{-N}^{N}\,dF(s)\int_{|u|>N}u^2\varphi_{\sigma}(u)\,du +
2\int_{|s|>N}\,dF(s)\int_{-\infty}^{\infty}u^2\varphi_{\sigma}(u)\,du.
\ene
It remains to apply the previous step together with the estimate 
$\int_N^{\infty}u^2\varphi_{\sigma}(u)\,du\le 
c\sigma N e^{-N^2/(2\sigma^2)}$.
The same estimate holds for $\int_{|x|>2N}x^2\,dF^*_{\sigma}(x)$.
\qed

\section{Entropic distance to normal laws for regularized random variables}

We keep the same notations as in the previous section and use the 
relations (2.1) and (2.4) when needed. In this section we obtain some results 
about the regularized random variables $X_{\sigma}$ and $X^*_{\sigma}$, 
which also hold for $Y_{\sigma}$ and $Y^*_{\sigma}$.
Denote by $p_{X_{\sigma}}$ and $p_{X^*_{\sigma}}$ the (smooth positive) 
densities of $X_{\sigma}$ and $X^*_{\sigma}$, respectively.

\vskip2mm
\begin{lemma}\label{lem3}
With some absolute constant $C$ we have, for all $x \in \R$,
\begin{align}
|p^{(k)}_{X_{\sigma}}(x) - p^{(k)}_{X^*_{\sigma}}(x)| \le 
C\sigma^{-2k-1}\sqrt{\ep},\quad k = 0,1,2.
\label{lem3.1}
\end{align} 
\end{lemma}

\begin{proof}
Write
\begin{align}
p_{X_{\sigma }}(x)
 = \int_{-\infty}^{\infty}\varphi_{\sigma}(x-s)\,dF(s)&=
\int_{-N}^{N}\varphi_{\sigma}(x-s)\,dF(s)
+\int_{|s|>N}\varphi_{\sigma}(x-s)\,dF(s),\notag\\
p_{X^*_{\sigma}}(x)
 = \int_{-\infty}^{\infty}\varphi_{\sigma}(x-s)\,dF^*(s)& =
\int_{-N}^{N}\varphi_{\sigma}(x-s)\,dF(s)\notag\\
&+(1-F(N)+F((-N)-)\, \varphi_{\sigma}(x).\notag
\end{align}
Hence
\begin{equation}\notag
|p_{X_{\sigma} }(x)-p_{X^*_{\sigma}}(x)| \le 
\frac{1}{\sqrt{2\pi}\sigma}\,\big(1-F(N)+F(-N)\big).
\end{equation}
But, by Lemma~\ref{lem2'}, and recalling the definition of $N = N(\ep)$,
we have
$$
1-F(N)+F(-N) \le 2(1-F_{\sigma}(N)+F_{\sigma}(-N)) \leq C\sqrt{\ep}
$$ 
with some absolute constant $C$. Therefore,
$|p_{X_{\sigma} }(x)-p_{X^*_{\sigma}}| \le C\sigma^{-1}\sqrt{\ep}$,
which is the assertion (\ref{lem3.1}) of the lemma in case $k=0$.
We obtain (\ref{lem3.1}) for $k=1,2$ in the same way. The lemma is proved.
\end{proof}

\begin{lemma}\label{lem4}
With some absolute constant $C>0$ we have
\begin{equation}\label{lem4.a}
D(X_{\sigma})\le D(X_{\sigma}^*) + C\sigma^{-3} N^3\sqrt{\ep}.
\end{equation} 
\end{lemma}

\begin{proof}
In general, if a random variable $U$ has density $u$ with finite 
variance $b^2$, then, by the very definition,
$$
D(U) =  \int_{-\infty}^{\infty} u(x)\log u(x)\,dx + \frac{1}{2}\,
\log(2\pi e\,b^2).
$$
Hence,
\begin{align} 
& D(X_{\sigma}) - D(X^*_{\sigma})=
\int_{-\infty}^{\infty} p_{X_{\sigma} }(x)\log p_{X_{\sigma} }(x)\,dx -
\int_{-\infty}^{\infty} p_{X^*_{\sigma} }(x)\log p_{X^*_{\sigma}}(x)\,dx\notag\\
& + \frac 12\log\frac{v_1^2+\sigma^2}{\sigma_1^2+\sigma^2}
=\int_{-\infty}^{\infty} (p_{X_{\sigma} }(x)-p_{X^*_{\sigma}}(x))
\log p_{X_{\sigma} }(x)\,dx\notag\\
& + \int_{-\infty}^{\infty}
p_{X^*_{\sigma}}(x)\log\frac{p_{X_{\sigma} }(x)}{p_{X^*_{\sigma}}(x)}\,dx
+\frac 12\log\frac{v_1^2+\sigma^2}{\sigma_1^2+\sigma^2}.\label{lem4.0}
\end{align}
Since $\E X^2 \leq 1$, necessarily $F(-2)+1-F(2)\le \frac{1}{2}$, so
\begin{equation}\label{lem4.1'}
\frac{1}{2\sigma\sqrt{2\pi}}\, e^{-(|x|+2)^2/(2\sigma^2)}\le 
p_{X^*_{\sigma}}(x)\le \frac{1}{\sigma\sqrt{2\pi}},
\end{equation}
and therefore
\begin{equation}\label{lem4.1}
|\log p_{X^*_{\sigma} }(x)|\le C\sigma^{-2}(x^2+4),\quad x\in\mathbb R,
\end{equation}
with some absolute constant $C$. The same estimate holds for 
$|\log p_{X_{\sigma}}(x)|$. 

Splitting the integration in
\begin{align} 
I_1 & =\int_{-\infty}^\infty 
(p_{X_{\sigma}}(x) - p_{X^*_{\sigma }}(x))\log p_{X_{\sigma}}(x)\,dx =
I_{1,1} + I_{1,2}
\notag\\
& = \Big(\int_{|x| \le 2N} + \int_{|x|>2N}\Big)
(p_{X_{\sigma}}(x)-p_{X^*_{\sigma}}(x))\log p_{X_{\sigma}}(x) \,dx,
\notag
\end{align}
we now estimate the integrals $I_{1,1}$ and $I_{1,2}$. 
By Lemma~\ref{lem3} and (\ref{lem4.1}), we get
$$
|I_{1,1}|\le C'\sigma^{-3}N^3\sqrt{\varepsilon}
$$
with some absolute constant $C'$. Applying (\ref{lem4.1}) together 
with Corollary 2.4, we also have
\bee 
|I_{1,2}| 
 & \le &
4C\sigma^{-2}\,
\big(1-F_{\sigma}(2N) + F_{\sigma}(-2N) + 1 - F^*_{\sigma}(2N) + F^*_{\sigma}(-2N)\big) \\
 & & + \ C\sigma^{-2}
\Big(\int_{|x|>2N}x^2\,dF_{\sigma}(x)+\int_{|x|>2N}x^2\,dF^*_{\sigma}(x)\Big)
 \ \le  \ C'\sigma^{-2}N^2\sqrt{\varepsilon}.
\ene
The two bounds yield
\begin{equation}\label{lem4.3}
|I_1|\le C''\sigma^{-3}N^3\sqrt{\varepsilon}
\end{equation}
with some absolute constant $C''$.

Now consider the integral
\begin{align}
I_2 
 & = \int_{-\infty}^\infty p_{X^*_{\sigma}}(x)
\log\frac{p_{X_{\sigma}}(x)}{p_{X^*_{\sigma} }(x)}\,dx =
I_{2,1}+I_{2,2}\notag\\
 & = \Big(\int_{|x|\le 2N}+\int_{|x|>2N}\Big)\ 
p_{X^*_{\sigma}}(x)\log\frac{p_{X_{\sigma}}(x)}{p_{X^*_{\sigma}}(x)}\,dx,\notag
\end{align}
which is non-negative, by Jensen's inequality.
Using $\log(1+t)\le t$ for $t\ge-1$, and Lemma~\ref{lem3}, we obtain
\bee
I_{2,1}
 & = &
\int_{|x|\le 2N}
p_{X^*_{\sigma}}(x)\log\Big(1+\frac{p_{X_{\sigma}}(x)-p_{X^*_{\sigma}}(x)}
{p_{X^*_{\sigma}}(x)}\Big)\,dx \\
 & \le & 
\int_{|x|\le 2N}|p_{X_{\sigma }}(x)-p_{X^*_{\sigma}}(x)|\,dx \ \le \ 
4C\sigma^{-1} N\sqrt{\ep}.
\ene
It remains to estimate $I_{2,2}$. We have as before, using (\ref{lem4.1}) 
and Corollary 2.4,
$$
|I_{2,2}|\le 
C\int_{|x|>2N}p_{X^*_{\sigma}}(x) \frac{x^2+4}{\sigma^2}\,dx\le
C'\sigma^{-2}N^2\sqrt{\ep}
$$
with some absolute constant $C'$. These bounds yield 
\be
I_2 \le C''\sigma^{-2}N^2\sqrt{\ep}.
\en

In addition, by Lemma~\ref{lem1},
$$
\log\frac{v_1^2+\sigma^2}{\sigma_1^2+\sigma^2} \leq 
\frac{v_1^2 - \sigma_1^2}{\sigma^2} \le C\sigma^{-2}N^2\sqrt{\ep}. 
$$
It remains to combine this bound with (3.6)-(3.7) and apply them in (3.3).
\end{proof}

\section{Fisher information for regularized random variables}

To compare the standardized Fisher information of $X_\sigma$ and 
$X_\sigma^*$, we need the following simple lemmas. These lemmas hold 
for both random variables $Y_{\sigma}$ and $Y^*_{\sigma}$.

\begin{lemma}\label{lem5a}
For $j=0,1,2$,
\begin{align}
&p^{(j)}_{X_{\sigma}}(x)(x^2+1)\to 0,\quad p^{(j)}_{X^*_{\sigma}}(x)(x^2+1)\to 0,
\quad\text{as}\quad x\to\pm\infty;\label{lem5a.1}\\
&\int_{-\infty}^{\infty}|p^{(j)}_{X_{\sigma}}(x)|(x^2+1)\,dx<\infty,
\quad
\int_{-\infty}^{\infty}|p^{(j)}_{X^*_{\sigma}}(x)|(x^2+1)\,dx<\infty.
\label{lem5a.2}
\end{align}
\end{lemma}
\begin{proof}
In view of the formula $p^{(j)}_{X_{\sigma}}(x)=
\int_{-\infty}^{\infty}\varphi_{\sigma}^{(j)}(x-s)\,dF(s),\,x\in\mathbb R$, 
we have the simple upper bound, for all $x\in\mathbb R$,
\begin{align}
|p^{(j)}_{X_{\sigma}}(x)|&\le\sigma^{-2}\Big(\int_{|x-s|< |x|/2}+
\int_{|x-s|\ge 
|x|/2}\Big)\Big(1+\frac{(x-s)^2}{\sigma^2}\Big)\varphi_{\sigma}(x-s)\,dF(s) \notag\\
&\le C\sigma^{-2}\Big(1-F(|x|/2)+F(-|x|/2)+
\Big(1+\frac{(|x|/2)^2}{\sigma^2}\Big)\varphi_{\sigma}(|x|/2)\Big),\quad j=0,1,2,\notag 
\end{align}
with some absolute constant $C>0$. Since $\int_{-\infty}^{\infty}s^2\,dF(s)<\infty$, 
the first relation of (\ref{lem5a.1}) follows immediately from this bound. 
We prove the second relation of (\ref{lem5a.1}) in the same way. 

Let us prove the first inequality of (\ref{lem5a.2}). We have
\begin{align}
&\int_{-\infty}^{\infty}|p^{(j)}_{X_{\sigma}}(x)|(x^2+1)\,dx\le
\int_ {-\infty}^{\infty}(x^2+1)
\bigg[\int_{-\infty}^{\infty}|\varphi_{\sigma}^{(2)}(x-s)|\,dF(s)\bigg]\,dx\notag\\
 &\le 
2\int_{-\infty}^{\infty}s^2\,dF(s) 
\int_{-\infty}^{\infty}|\varphi_{\sigma}^{(2)}(u)|\,du + 2\int_{-\infty}^{\infty}\,dF(s)
\int_{-\infty}^{\infty}(u^2+1) |\varphi_{\sigma}^{(2)}(u)|\,du <\infty.\notag
\end{align}
The second relation in (\ref{lem5a.2}) follows similarly.
\end{proof}

\begin{lemma}\label{lem5a'}
With some absolute constant $C$
\begin{align}
\int_{|x|>2N}|p''_{X_{\sigma}}(x)|(x^2+1)\,dx \leq 
C\sigma^{-2}N^2\sqrt{\ep},\,\,
\int_{|x|>2N}|p''_{X^*_{\sigma}}(x)|(x^2+1)\,dx \leq 
C\sigma^{-2}N^2\sqrt{\ep}. \notag
\end{align}
\end{lemma}
\begin{proof}

Note that
\begin{align}
&\int_{|x|>2N} |p''_{X_{\sigma}}(x)|(x^2+1)\,dx 
\le
\int_{-\infty}^{\infty} dF(s)
\int_{|x|>2N}(x^2+1)|\varphi_{\sigma}^{(2)}(x-s)|\,dx \notag\\
 &\le 
2\int_{-N}^{N}s^2\,dF(s) \int_{|u|>N}|\varphi_{\sigma}^{(2)}(u)|\,du + 
2\int_{|s|>N}s^2\,dF(s)
\int_{-\infty}^{\infty} |\varphi_{\sigma}^{(2)}(u)|\,du \notag\\
&+2\int_{-N}^{N}\,dF(s)\int_{|u|>N}(u^2+1)|\varphi_{\sigma}^{(2)}(u)|\,du +
2\int_{|s|>N}\,dF(s)\int_{-\infty}^{\infty}
(u^2+1)|\varphi_{\sigma}^{(2)}(u)|\,du.\label{lem5a'.1}
\end{align}
Since 
\begin{equation}
|\varphi_{\sigma}^{(2)}(u)|\le 
\frac 1{\sigma^2}\Big(1+\frac{u^2}{\sigma^2}\Big)\varphi_{\sigma}(u),
\quad u\in\mathbb R,\notag
\end{equation}
we get the following estimates
\begin{align}
 &\int_{|u|>N}|\varphi_{\sigma}^{(2)}(u)|\,du\le 
\frac 1{\sigma^2}\int_{|u|>N}\Big(1+\frac{u^2}{\sigma^2}\Big)\varphi_{\sigma}(u)\,du
 \le C\frac N{\sigma^3}e^{-N^2/(2\sigma^2)}\le \frac C{\sigma^2}\varepsilon,\notag\\
 &\int_{|u|>N}(u^2+1)|\varphi_{\sigma}^{(2)}(u)|\,du\le 
\frac 1{\sigma^2}\int_{|u|>N}\Big(1+\frac{u^2}{\sigma^2}\Big)^2\varphi_{\sigma}(u)\,du
 \le C\frac{N^3}{\sigma^5}e^{-N^2/(2\sigma^2)}\le \frac C{\sigma^2}N^2\varepsilon.\notag
\end{align}
Applying these upper bounds and Collorary~2.4 to (\ref{lem5a'.1}) we easily obtain
the first assertion of the lemma. We may prove the second assertion of the 
lemma in the same way.
\end{proof}

The next representations are well-known; they are obtained, using Lemma~\ref{lem5a} and the bound (\ref{lem4.1}), via integration 
by parts in the integral which defines the Fisher information.

\begin{lemma}\label{lem5b}
The following formulas hold
\begin{align}
I(X_{\sigma})=-\int_{-\infty}^{\infty}
p''_{X_{\sigma} }(x)\log p_{X_{\sigma} }(x)\,dx,\quad
I(X^*_{\sigma} )=-\int_{-\infty}^{\infty}
p''_{X^*_{\sigma }}(x)\log p_{X^*_{\sigma}}(x)\,dx.
\end{align}
\end{lemma}

We are now prepared to prove the following bound.

\begin{lemma}\label{lem5c}
With some absolute constant $C>0$ we have
\begin{align}
(v_1^2+\sigma^2)^{-1}J_{st}(X_{\sigma}) \, \le \, 
(\sigma_1^2+\sigma^2)^{-1} J_{st}(X^*_{\sigma})
+ C\sigma^{-7} N^3\sqrt{\ep}.\notag
\end{align}  
\end{lemma}

\begin{proof}
Write
\begin{align} 
 & \frac{J_{st}(X_{\sigma})}{v_1^2+\sigma^2} -
\frac{J_{st}(X^*_{\sigma})}{\sigma_1^2+\sigma^2} = 
-\int_{-\infty}^\infty p_{X_{\sigma}}''(x)\log p_{X_{\sigma}}(x)\,dx
+\int_{-\infty}^\infty p_{X^*_{\sigma}}''(x)\log p_{X^*_{\sigma}}(x)\,dx\notag\\
 & +
\frac{v_1^2-\sigma_1^2}{(v_1^2+\sigma^2)(\sigma_1^2+\sigma^2)} =
-\int_{-\infty}^\infty (p_{X_{\sigma}}''(x)-p_{X^*_{\sigma}}''(x))
\log p_{X_{\sigma}}(x)\,dx\notag\\
 & -
\int_{-\infty}^\infty p_{X^*_{\sigma}}''(x)
\log\frac{p_{X_{\sigma} }(x)}{p_{X^*_{\sigma}}(x)}\,dx +
\frac{v_1^2-\sigma_1^2}{(v_1^2+\sigma^2)(\sigma_1^2+\sigma^2)}.
\label{lem5c.0}
\end{align} 
Splitting integration in the first integral on the right-hand side of (\ref{lem5c.0}),
\begin{align} 
J_1 & =
\int_{-\infty}^\infty 
(p''_{X_{\sigma}}(x) - p''_{X^*_{\sigma}}(x))\log p_{X_{\sigma}}(x) \,dx = J_{1,1}+J_{1,2}
\notag\\
 & =
\Big(\int_{|x|\le 2N} + \int_{|x|>2N}\Big)
(p''_{X_{\sigma}}(x)-p''_{X^*_{\sigma}}(x))\log p_{X_{\sigma}}(x) \,dx,
\notag
\end{align}
we now estimate the integrals $J_{1,1}$ and $J_{1,2}$. 

By Lemma~\ref{lem3} and (\ref{lem4.1}),
\begin{equation}\label{lem5c.1} 
|J_{1,1}|\le C'\sigma^{-7}N^3\sqrt{\varepsilon}
\end{equation}
with some absolute constant $C'$, while,
by (\ref{lem4.1}) and Lemma~\ref{lem5a'},
\begin{align} 
|J_{1,2}| \le C'\sigma^{-4}N^2\sqrt{\varepsilon}.\label{lem5c.2}  
\end{align}
Now consider the integral
\begin{align}
J_2 & = 
\int_{-\infty}^\infty p''_{X^*_{\sigma}}(x)
\log\frac{p_{X_{\sigma }}(x)}{p_{X^*_{\sigma}}(x)}\,dx =
 J_{2,1}+J_{2,2}\notag\\
 & =
\Big(\int_{|x|\le 2N}+\int_{|x|>2N}\Big) \ p''_{X^*_{\sigma}}(x)
\log\frac{p_{X_{\sigma}}(x)}{p_{X^*_{\sigma}}(x)}\,dx.\notag
\end{align}
It is easy to verify that
\begin{equation}
p''_{X^*_{\sigma}}(x)=-\sigma^{-2}p_{X^*_{\sigma}}(x) +
\sigma^{-4}\int_{-\infty}^\infty (x-s)^2\varphi_{\sigma}(x-s)dF^*(s).
\notag
\end{equation}
Therefore, for $x\in[-2N,2N]$,
\begin{equation}
|p''_{X^*_{\sigma}}(x)| \, \le \, 
\sigma^{-2}(1+9\sigma^{-2}N^2)\, p_{X^*_{\sigma}}(x),\notag
\end{equation}
which leads to the upper bound 
\begin{align}
|J_{2,1}|\le \sigma^{-2} (1+9\sigma^{-2} N^2)
\Big(\int_{E_-} p_{X^*_{\sigma}}(x)
\log\frac{p_{X^*_{\sigma}}(x)}{p_{X_{\sigma}}(x)}\,dx +
\int_{E_+} p_{X^*_{\sigma}}(x)
\log\frac{p_{X_{\sigma}}(x)}{p_{X^*_{\sigma}}(x)}\,dx\Big),\label{lem5c.3}
\end{align}
where $E_+ = \{x\in[-2N,2N]: p_{X_{\sigma}}(x)>p_{X^*_{\sigma}}(x)\}$ and 
$E_- = [-2N,2N]\setminus E_+$. As in the proof of the estimate on $I_{2,1}$ 
in the previous section, we obtain
\begin{equation}
\int_{E_+}p_{X^*_{\sigma}}(x)
\log\frac{p_{X_{\sigma}}(x)}{p_{X^*_{\sigma}}(x)}\,dx \le 
4C\sigma^{-1}N\sqrt{\ep}.
\label{lem5c.4}
\end{equation}
From (\ref{lem4.1'}) and from the definition of $p_{X_{\sigma}}(x)$ and 
$p_{X^*_{\sigma}}(x)$, we see that, for $x\in E_-$,
\begin{equation}
\frac{1}{10}\, p_{X^*_{\sigma}}(x)\ge 
\big(1-F(N)+F((-N)-)\big)\,\varphi_{\sigma}(x) \ge \int_{|s|\ge N}
\varphi_{\sigma}(x-s)\,dF(s)\notag
\end{equation}
and we have $p_{X^*_{\sigma}}(x)/p_{X_{\sigma}}(x)\le 2$ for $x\in E_-$. 
Therefore, as above, we get
\begin{equation}
\int_{E_-} p_{X^*_{\sigma}}(x)
\log\frac{p_{X^*_{\sigma}}(x)}{p_{X_{\sigma}}(x)}\,dx \le
2\int_{E_-} p_{X_{\sigma}}(x) 
\log\frac{p_{X^*_{\sigma}}(x)}{p_{X_{\sigma}}(x)}\,dx \le 
8C\sigma^{-1} N\sqrt{\ep}.
\label{lem5c.5}
\end{equation}
Applying (\ref{lem5c.4}) and (\ref{lem5c.5}) to (\ref{lem5c.3}), 
we finally obtain
\begin{equation}
|J_{2,1}|\le C'\sigma^{-5} N^3\sqrt{\ep}\label{lem5c.6}
\end{equation}
with some absolute constant $C'$.

It remains to estimate $J_{2,2}$. We have, using (\ref{lem4.1}) and 
Lemma~\ref{lem5a'},
\begin{equation}
|J_{2,2}|\le C\int_{|x|>2N} p''_{X^*_{\sigma }}(x) (x^2+c^2)\,
\sigma^{-2}\,dx\le C''\sigma^{-4}N^2\sqrt{\varepsilon}  \label{lem5c.7} 
\end{equation}
with some absolute constant $C''$. Applying  Lemma~\ref{lem1},
(\ref{lem5c.1})--(\ref{lem5c.2}) and (\ref{lem5c.6})--(\ref{lem5c.7}) in the 
representation (\ref{lem5c.0}), we arrive at the assertion of the lemma.
\end{proof}

\section{Characteristic functions of truncated random variables}

Denote by $f_{X^*}(t)$ and $f_{Y^*}(t)$ the characteristic functions of 
the random variables $X^*$ and $Y^*$, respectively. As integrals over 
finite intervals they admit analytic continuations as entire functions 
to the whole complex plane $\mathbb C$. These continuations will be denoted
by $f_{X^*}(t)$ and $f_{Y^*}(t)$, ($t\in\mathbb C$). 

Put $T = \frac 1{64}\,N = \frac 1{64}\,
\sqrt{1+2\sigma^2}\,\big(1+\sqrt{2\log\frac{1}{\ep}}\,\big)$. We assume 
without loss of generality that $0 < \ep \le \ep_0$, where $\ep_0$ is 
a sufficiently small absolute constant.

\begin{lemma}\label{lem5d}
For all $t\in\mathbb C,\,|t|\le T$,
\begin{equation}
\frac{1}{2}\,|e^{-t^2/2}|\le |f_{X^*}(t)|\,|f_{Y^*}(t)|\le
\frac{3}{2}\,|e^{-t^2/2}|. \label{lem5d.1}
\end{equation} 
\end{lemma}

\begin{proof}
For all complex $t$,
\begin{align}
 \Big|\int_{-\infty}^\infty e^{itx}\,d(F^*_{\sigma}*G^*_{\sigma})&(x) -
\int_{-\infty}^\infty e^{itx}\,d\Phi_{\sqrt{1+2\sigma^2}}(x)\Big| \le
\Big|\int_{-4N}^{4N}
e^{itx}\,d(F^*_{\sigma}*G^*_{\sigma}-\Phi_{\sqrt{1+2\sigma^2}})(x)\Big|\notag\\
 & +
\int_{|x|\ge 4N}e^{-x\Im(t)}\,d(F^*_{\sigma}*G^*_{\sigma})(x)+\int_{|x|\ge 4N}
e^{-x\Im(t)}\,\varphi_{\sqrt{1+2\sigma^2}}(x)\,dx.\label{lem5d.2}
\end{align}
Integrating by parts, we have
\begin{align}
& \int_{-4N}^{4N}
e^{itx}\,d(F^*_{\sigma}*G^*_{\sigma}-\Phi_{\sqrt{1+2\sigma^2}})(x) \ = \ 
e^{4itN}(F^*_{\sigma}*G^*_{\sigma} -\Phi_{\sqrt{1+2\sigma^2}})(4N)\notag\\
&-e^{-4itN} (F^*_{\sigma}*G^*_{\sigma}-
\Phi_{\sqrt{1+2\sigma^2}})(-4N)-it\int_{-4N}^{4N}(F^*_{\sigma}*G^*_{\sigma}-
\Phi_{\sqrt{1+2\sigma^2}})(x)\,e^{itx}\,dx.\notag
\end{align}
In view of the choice of $T$ and $N$, we easily obtain, using Lemma~\ref{lem2}, 
for all $|t|\le T$, 
\begin{equation}\label{lem5d.3}
\big|\int_{-4N}^{4N}e^{itx}\,
d(F^*_{\sigma}*G^*_{\sigma}-\Phi_{\sqrt{1+2\sigma^2}})(x)\Big| \le
2C\sqrt{\ep}\,e^{4N|\Im(t)|}+8C|t|\sqrt{\ep}\,
e^{4N\,|\Im(t)|}\le \frac{1}{6}\, e^{-(1/2+\sigma^2)\,T^2}.
\end{equation}
The second integral on the right-hand side of (\ref{lem5d.2}) does not 
exceed, for $|t|\le T$,
\begin{align}
 & \int_{-2N}^{2N} d(F^* * G^*)(s) \int_{|x|\ge 4N}
e^{-x\,\Im(t)}\varphi_{\sqrt 2\sigma}(x-s)\,dx\notag\\
 & \le\int_{-2N}^{2N} e^{-s\,\Im(t)}\,d(F^**G^*)(s)\cdot
 \int_{|u|\ge 2N} e^{-u\,\Im(t)}\varphi_{\sqrt 2\sigma}(u)\,du\notag\\
 & \le \, e^{2NT}\cdot\frac 1{\sqrt{\pi}}\int_{2N/\sigma}^{\infty}
 e^{\sigma Tu-u^2/4}\,du \ \le \ \frac{1}{6}\, e^{-(1/2+\sigma^2) T^2}.
\label{lem5d.4}
\end{align}
The third integral on the right-hand side of (\ref{lem5d.2}) does not 
exceed, for $|t|\le T$,
\begin{equation}\label{lem5d.5}
\sqrt{\frac2{\pi}}\int_{4N}^{\infty}e^{Tu-u^2/6}\,du
 \, \le \, \frac{1}{6}\, e^{-(1/2+\sigma^2)\,T^2}.
\end{equation}
Applying (\ref{lem5d.3})-(\ref{lem5d.5}) in (\ref{lem5d.2}), we arrive 
at the upper bound
\begin{equation}
|e^{-\sigma^2t^2/2} f_{X^*}(t) e^{-\sigma^2t^2/2} f_{Y^*}(t) -
e^{-(1/2+\sigma^2)t^2}| \le 
\frac{1}{2}\,e^{-(1/2+\sigma^2)T^2}\le 
\frac{1}{2}\,|e^{-(1/2+\sigma^2)t^2}|
\end{equation}
from which (\ref{lem5d.1}) follows.
\end{proof}

The bounds in (\ref{lem5d.1}) show that the characteristic function 
$f_{X^*}(t)$ does not vanish in the circle $|t|\le T$.
Hence, using results from ([L-O], pp. 260--266), we conclude that $f_{X^*}(t)$ 
has a representation
\begin{equation}\notag
f_{X^*}(t)=\exp\{g_{X^*}(t)\},\quad g_{X^*}(0)=0,
\end{equation}
where $g_{X^*}(t)$ is analytic on the circle $|t|\le T$ and admits 
the representation
\begin{equation}\label{lem5d.6}
g_{X^*}(t)=ia_1t-\frac12 \sigma_1^2 t^2-\frac 12t^2\psi_{X^*}(t),
\end{equation}
where 
\begin{equation}\label{lem5d.7}
\psi_{X^*}(t)=\sum_{k=3}^{\infty}i^kc_k\Big(\frac tT\Big)^{k-2}
\end{equation}
with real-valued coefficients $c_k$ such that $|c_k|\le C$ for some 
absolute constant $C$. In the sequel without loss of generality we assume 
that $a_1=0$. An analogous representation holds for the function 
$f_{Y^*}(t)$.

\section{Derivatives of the density of the random variable $X^*_{\sigma}$}

We shall use the following inversion formula
\begin{equation}\notag
p^{(k)}_{X^*_{\sigma}}(x) = \frac{1}{2\pi}
\int_{-\infty}^{\infty}(-it)^k e^{-ixt}
e^{-\sigma^2t^2/2}f_{X^*}(t)\,dt,
\quad k=0,1,\dots,\quad x\in\mathbb R,
\end{equation}
for the derivatives of the density $p_{X^*_{\sigma}}(x)$. By Cauchy's 
theorem, one may change the path of integration in this integral 
from the real line to any line $z=t+iy,\,t\in\mathbb R$, with parameter
$y\in\mathbb R$. This results in the following representation
\begin{equation}\label{6.1}
p^{(k)}_{X^*_{\sigma}}(x) =
e^{yx}e^{\sigma^2y^2/2}f_{X^*}(iy)\cdot I_k(x,y),\quad x\in\mathbb R.
\end{equation}
Here
\begin{equation}\label{6.1*}
I_k(x,y)=\frac 1{2\pi} \int_{-\infty}^{\infty}(-i(t+iy))^kR(t,x,y)\,dt,
\end{equation}
where
\begin{equation}\label{6.2}
R(t,x,y) = f_{X^*}(t+iy)e^{-it(x+\sigma^2y)-\sigma^2t^2/2}/f_{X^*}(iy).
\end{equation}

Let us now describe the choice of the parameter $y\in\mathbb R$ in (\ref{6.1}). 
It is well-known, that the function 
$\log f_{X^*}(iy),y\in\mathbb R$, is convex. Therefore, the function 
$\frac d{dy}\log f_{X^*}(iy)+\sigma^2y$ is strictly monotone and tends to 
$-\infty$ as $y\to-\infty$ and tends to $\infty$ as $y\to\infty$. By 
(\ref{lem5d.6}) and (\ref{lem5d.7}), this function is vanishing at zero.
Hence, the equation
\begin{equation}\label{6.3}
\frac d{dy}\log f_{X^*}(iy)+\sigma^2y = -x 
\end{equation}
has a unique continuous solution $y=y(x)$ such that $y(x)<0$ for $x>0$ and 
$y(x)>0$ for $x<0$. Here and in the sequel we use the principal branch of 
$\log z$.

We shall need one representation of the solution $y(x)$ in the interval 
$x\in[-(\sigma_1^2+\sigma^2)T_1,(\sigma_1^2+\sigma^2)T_1]$,
where $T_1 = c'(\sigma_1^2+\sigma^2)T$ with a sufficiently small 
absolute constant $c'>0$. We see that
\begin{align}
q_{X^*}(t) & \equiv 
\frac{d}{dt}\,\log f_{X^*}(t)-\sigma^2t =
- (\sigma_1^2+\sigma^2) t - r_1(t) - r_2(t)\notag\\
& = -(\sigma_1^2+\sigma^2) t - t \psi_{X^*}(t) - 
\frac{1}{2}\, t^2\psi_{X^*}'(t).\label{6.4}
\end{align}
The functions $r_1(t)$ and $r_2(t)$ are analytic in the circle 
$\{|t|\le T/2\}$ and there, by (\ref{lem5d.7}), they may be bounded as follows
\begin{equation}\label{6.5} 
|r_1(t)|+|r_2(t)|\le C|t|^2/T
\end{equation}
with some absolute constant $C$. 
Using (\ref{6.4}), (\ref{6.5}) and Rouch\'e's theorem, we conclude that 
the function $q_{X^*}(t)$ is univalent in the circle $D = \{|t|\le T_1\}$, 
and $q_{X^*}(D)\supset \frac 12(\sigma_1^2+\sigma^2)D$. By the well-known 
inverse function theorem (see [S-G], pp. 159-160), we have
\begin{equation}\label{6.6} 
q_{X^*}^{(-1)}(w) = b_1 w + i b_2 w^2 -b_3 w^3+ \dots,  \qquad 
w\in \frac{1}{2}\,(\sigma_1^2+\sigma^2) D,
\end{equation}
where
\begin{equation}\label{6.7}
 i^{n-1}b_n=\frac 1{2\pi i}\int\limits_{|\zeta|=
\frac 12T_1}\frac{\zeta \cdot q_{X^*}'(\zeta)}
{q_{X^*}(\zeta)^{n+1}}\,d\zeta,\qquad n=1,2,\dots.
\end{equation}
Using this formula and (\ref{6.4})--(\ref{6.5}), we note that 
\begin{equation}\label{6.8}
b_1=-\frac 1{\sigma_1^2+\sigma^2}  
\end{equation}
and that all remaining coefficients $b_2,b_3,\dots$ are real-valued.
In addition, by (\ref{6.4})--(\ref{6.5}), 
\begin{equation}\notag
-\frac{q_{X^*}(t)}{(\sigma_1^2+\sigma^2)t}=1+q_1(t) \quad\text{and}\quad
-\frac{q_{X^*}'(t)}{\sigma_1^2+\sigma^2}=1+q_2(t), 
\end{equation}
where $q_1(t)$ and $q_2(t)$ are analytic functions in $D$ satisfying there 
$|q_1(t)|+|q_2(t)|\le \frac 12$. Therefore, for $\zeta\in D$,
\begin{equation}\notag 
\frac{q_{X^*}'(\zeta)}{q_{X^*}(\zeta)^{n+1}}=
(-1)^{n}\frac{q_3(\zeta)}{(\sigma_1^2+\sigma^2)^n\zeta^{n+1}}
 \equiv (-1)^{n}\frac{1+q_2(\zeta)}{(\sigma_1^2+\sigma^2)^n
(1+q_1(\zeta))^{n+1}\zeta^{n+1}},
\end{equation} 
where $q_3(\zeta)$ is an analytic function in $D$ such that 
$|q_3(\zeta)|\le 3\cdot2^n$. Hence, $q_3(\zeta)$ admits the representation 
\begin{equation}\notag 
q_3(\zeta)=1+\sum_{k=1}^{\infty}d_k\frac{\zeta^k}{T_1^k}
\end{equation}
with coefficients $d_k$ such that $|d_k|\le 3\cdot 2^n$. Using this 
equality, we obtain from (\ref{6.7}) that 
\begin{equation}\label{6.9}
b_n = \frac{d_{n-1}}{(\sigma_1^2+\sigma^2)^n\, T_1^{n-1}}\quad 
\text{and}\quad 
|b_n|\le \frac {3\cdot2^n}{(\sigma_1^2+\sigma^2)^n\,T_1^{n-1}},\quad n=2,\dots.   
\end{equation}
Now we can conclude from (\ref{6.6}) and (\ref{6.9}) that, for 
$|x|\le T_1/(4|b_1|)$,
\begin{equation}\label{6.10}
y(x)=-iq_{X^*}^{(-1)}(ix) = 
b_1 x - b_2 x^2 + R(x),\quad\text{where}\quad |R(x)|\le 
48\,|b_1|^3|x|^3/T_1^2.
\end{equation}

In the sequel we denote by $\theta$ a real-valued quantity such that 
$|\theta|\le 1$. Using (\ref{6.10}), let us prove:

\begin{lemma}\label{lem6.1}
In the interval $|x|\le c''T_1/|b_1|$ with a
sufficiently small positive absolute constant $c''$,
\begin{equation}\label{6.11}
y(x)x + \frac{1}{2}\,\sigma^2y(x)^2 + \log f_{X^*}(iy(x)) =
\frac{1}{2}\, b_1x^2 + \frac{c_3b_1^3}{2T}\, x^3 + 
\frac{c\theta b_1^5}{T^2}\,x^4,
\end{equation}
where $c$ is an absolute constant.
\end{lemma} 
\begin{proof}
From (\ref{6.9}) and (\ref{6.10}), it follows that
\begin{equation}\label{lem6.11'}
\frac 12|b_1x|\le |y(x)|\le \frac 32|b_1x|.
\end{equation}
Therefore,
\begin{align}
\frac 12 y(x)^2\sum_{k=4}^{\infty}|c_k|\Big(\frac{|y(x)|}T\Big)^{k-2}
\le C\Big(\frac 32\Big)^4\frac{b_1^4x^4}{T^2}.\notag
\end{align}
On the other hand, with the help of (\ref{6.9}) and (\ref{6.10}) one can 
easily deduce the relation
\begin{align}
y(x)x + \frac{1}{2}\,(\sigma^2 + \sigma_1^2)\, y(x)^2 +
\frac{1}{2}\, c_3\frac{y(x)^3}T \ = \ 
\frac{1}{2}\, b_1 x^2 + \frac{1}{2}\,c_3 b_1^3\frac{x^3}{T}
+\frac{c\theta b_1^5}{T^2}\,x^4 \notag
\end{align}
with some absolute constant $c$. The assertion of the lemma follows 
immediately from the two last relations. 
\end{proof}

Now, applying Lemma~\ref{lem6.1} to (\ref{6.1}), we may conclude that in the interval 
$|x|\le c''T_1/|b_1|$, the derivative $p^{(k)}_{X^*_{\sigma}}(x)$ 
admits the representation
\begin{equation}\label{6.12}
p^{(k)}_{X^*_\sigma}(x) = 
\exp\Big\{\frac{1}{2}\, b_1 x^2 + \frac{1}{2}\, c_3 b_1^3\, \frac{x^3}{T} +
\frac{c\theta b_1^5}{T^2}\,x^4\Big\} \cdot I_k(x,y(x))
\end{equation}
with some absolute constant $c$.

As for the values $|x|>c''T_1/|b_1|$, in (\ref{6.1}) we choose 
$y=y(x) =y(c''T_1/|b_1|)$ for $x>0$ and $y=y(x)=y(-c''T_1/|b_1|)$ for $x<0$. 
In this case, by (\ref{lem6.11'}), we note that
$|y|\le 3c''T_1/2$, and we have
\begin{align}
\Big|\frac{1}{2}\, \sigma^2 y^2 & + \log f_{X^*}(iy)\Big| \, \le \, 
\frac{y^2}{2|b_1|} +
\frac{C}{2}\,\frac{|y|^3}{T}
\sum_{k=3}^{\infty}\Big(\frac{|y|}T\Big)^{k-3}\notag\\
 & \, \le \,
\frac{1}{2}\,|y|\,\bigg[\frac{3c''T_1}{2|b_1|} + \frac{9}{4}\,C (c'')^2\, 
\frac{T_1^2}{T}\sum_{k=3}^{\infty}
\Big(\frac{3c''T_1}{2T}\Big)^{k-3}\bigg]
 \le \frac{1}{2}\,|y|\,
 \Big(\frac{3}{2}\,|x| + \frac{1}{4}\,|x|\Big) \le 
 \frac{7}{8}\,|yx|.\notag
\end{align}
As a result, we obtain from (\ref{6.1}) an upper bound
$|p^{(k)}_{X^*_{\sigma}}(x)|\le e^{-\frac 18|y(x)x|}\,|I_k(x,y(x))|$ for 
$|x|>c''T_1/|b_1|$, 
which with the help of left-hand side of (\ref{lem6.11'}) yields the estimate
\begin{equation}\label{6.13}
|p_{X^*_{\sigma}}^{(k)}(x)|\le e^{-cT|x|/|b_1|}|I_k(x,y(x))|,\quad 
|x|>c''T_1/|b_1|,
\end{equation}
with some absolute constant $c>0$.

\section{The estimate of the integral $I_0(x,y)$}

In order to study the behavior of the integrals $I_k(x,y)$, we need some 
auxiliary results. We use the letter $c$ to denote absolute constants which
may vary from place to place.

\begin{lemma}\label{lem7.1}
For $t,y\in[-T/4,T/4]$ and $x\in\mathbb R$, we have the relation
\begin{equation}\label{lem7.1.1}
\log|R(t,x,y)|=-\gamma(y)t^2/2+r_1(t,y),
\end{equation}
where
\begin{equation}\label{lem7.1.1'}
\gamma(y) = |b_1|^{-1}+\psi_{X^*}(iy)+2iy\psi'_{X^*}(iy)
\end{equation}
and
\begin{equation}\label{lem7.1.1''}
|r_1(t,y)|\le ct^2(t^2+y^2)T^{-2}\quad\text{with some absolute constant}
\,\, c.
\end{equation}
\end{lemma}

\begin{proof}
From the definition of the function $R(t,x,y)$ it follows that
\begin{align}
\log|R(t,x,y)|&=\frac 12\Big(\frac 1{b_1}-\psi_{X^*}(iy)-2iy\psi'_{X^*}(iy)\Big)t^2-
\frac 12(\Re\psi_{X^*}(t+iy)-\psi_{X^*}(iy))(t^2-y^2)
\notag\\
&+(\Im\psi_{X^*}(t+iy)+it\psi'_{X^*}(iy))ty.\label{lem7.1.2'}
\end{align}
Since, for $t,y\in[-T/4,T/4]$ and $k=4,\dots$, 
\begin{align}
 \left|\Re(i^k(t+iy)^{k-2}-i^k(iy)^{k-2})\right| &=
\bigg|\sum_{l=0}^{(k-2)/2} (-1)^{k+1+l}{k-2\choose 2l} t^{2l}y^{k-2-2l}-
(-1)^{k+1}y^{k-2}\bigg|\notag\\
&\le t^2(T/4)^{k-4}\sum_{l=1}^{(k-2)/2}{k-2\choose 2l}\le 4t^2(T/2)^{k-4},\notag
\end{align}
we obtain an upper bound, for the same $t$ and $y$, namely
\begin{equation}\label{lem7.1.3}
|\Re\psi_{X^*}(t+iy)-\psi_{X^*}(iy)|\le 
\sum_{k=4}^{\infty}\frac{|c_k|}{T^{k-2}}|\Re(i^k(t+iy)^{k-2}-i^k(iy)^{k-2})|
\le \frac{2^3Ct^2}{T^2}.
\end{equation}
Since, for $t,y\in[-T/4,T/4]$ and $k=5,\dots$, 
\begin{align}
 & \left|\,\Im(i^k(t+iy)^{k-2}-i^{k}(k-2)t(iy)^{k-3})\right| =
\bigg|\sum_{l=1}^{(k-3)/2}{k-2\choose 2l+1}(-1)^{k+l}t^{2l+1}y^{k-3-2l}\bigg|
\notag\\
 & \le |t|^3(T/4)^{k-5}\sum_{l=1}^{(k-3)/2}{k-2\choose 2l+1}\le 8|t|^3(T/2)^{k-5},\notag
\end{align}
we have
\begin{equation}\label{lem7.1.4}
|\,\Im\psi_{X^*}(t+iy)+it \psi_{X^*}'(iy)|\le \sum_{k=5}^{\infty}
\frac{|c_k|}{T^{k-2}}|\Im(i^k(t+iy)^{k-2}-ti^k(k-2)(iy)^{k-3})|
\le \frac{2^4C|t|^3}{T^3}
\end{equation}
for the same $t$ and $y$. Applying (\ref{lem7.1.3}) and (\ref{lem7.1.4}) 
in (\ref{lem7.1.2'}), we obtain the assertion of the lemma.
\end{proof}
\begin{lemma}\label{lem7.1a}
For $|t|\le c''T/\sqrt{|b_1|}$ and $|y|\le c''T/|b_1|$, we have the 
estimates
\begin{equation}\label{lem7.1a.1}
\frac{3}{4|b_1|}\le \gamma(y) \le \frac{5}{4|b_1|}
\end{equation}
and
\begin{equation}\label{lem7.1a.2}
|r_1(t,y)|\le t^2/(8|b_1|).
\end{equation}
\end{lemma}
\begin{proof}
Recall that the positive absolute constant $c''$ is chosen to be 
sufficiently small. Using the following simple bounds
\begin{align}
|\psi_{X^*}(iy)|&\le \sum_{k=3}^{\infty}|c_k|\Big(\frac{|y|}T\Big)^{k-2}
\le C\frac{|y|}T\sum_{k=3}^{\infty}\Big(\frac{c''}{|b_1|}\Big)^{k-3}\le
\frac 1{8|b_1|},\label{lem7.1a.3}\\
2|y\psi'_{X^*}(iy)|&\le \frac{2|y|}T\sum_{k=3}|c_k|(k-2)\Big(\frac{|y|}T\Big)^{k-3}\le 
C\frac{2|y|}T\sum_{k=3}^{\infty}
(k-2)\Big(\frac{c''}{|b_1|}\Big)^{k-3}\le \frac 1{8|b_1|},\label{lem7.1a.4}
\end{align}
we easily obtain
\begin{align}
\frac 3{4|b_1|}&\le \frac 1{|b_1|}-\psi_{X^*}(iy)|-2|y\psi'_{X^*}(iy)|\le \gamma(y)\notag\\
&\le\frac 1{|b_1|}+|\psi_{X^*}(iy)|+2|y\psi'_{x^*}(iy)|\le \frac 5{4|b_1|},\notag
\end{align}
and thus (\ref{lem7.1a.1}) is proved. The bound (\ref{lem7.1a.2}) follows 
immediately from (\ref{lem7.1.1''}). 
\end{proof}

\begin{lemma}\label{lem7.2}
For $t\in[-T/4,T/4]$ and $x\in[-c''T_1/|b_1|,c''T_1/|b_1|]$, we have
\begin{equation}\label{lem7.2.1}
\Im\log R(t,x,y(x))=\frac i2 t^3\psi'_{X^*}(iy(x))+r_2(t,x),
\end{equation}
where
\begin{equation}\label{lem7.1.2}
|r_2(t,x)|\le c(|t|+|y(x)|)|t|^3T^{-2}\quad
\text{with some absolute constant}\,\,c.
\end{equation}
\end{lemma}

\begin{proof}
Write, for $t,y\in[-T/4,T/4]$ and $x\in\mathbb R$,
\begin{align}
\Im\log R(t,y,x)=-tx+\frac 1{b_1}ty-ty\Re\psi_{X^*}(t+iy)
-\frac 12(t^2-y^2)\Im\psi_{X^*}(t+iy).\label{lem7.2.3}
\end{align}
Now we choose in this formula $y=y(x)$, where $y(x)$ is a solution of 
the equation of (\ref{6.3}) for $x\in[-c''T_1/|b_1|,c''T_1/|b_1|]$. 
For such $x$, in view of (\ref{lem6.11'}), we know that $|y(x)|\le T/4$. 
Let us rewrite (\ref{6.3}) (see as well (\ref{6.4})) in the form
\begin{equation}\notag
-\frac 1{b_1}y(x)+y(x)\psi_{X^*}(iy(x))+\frac i2y^2\psi_{X^*}'(iy(x))) =-x. 
\end{equation}
Applying this relation in (\ref{lem7.2.3}), we obtain the formula
\begin{align}
&\Im\log R(t,x,y(x))=-ty(x)(\Re\psi_{X^*}(t+iy(x))-\psi_{X^*}(iy(x)))
+\frac i2t^3\psi_{X^*}'(iy(x))\notag\\
&-\frac 12(t^2-y(x)^2)\big(\Im\psi_{X^*}(t+iy(x))+it\psi_{X^*}'(iy(x))\big).\notag
\end{align}
In view of (\ref{lem7.1.3}) and (\ref{lem7.1.4}), we can conclude that
\begin{equation}\notag
\Im\log R(t,x,y(x))=\frac i2t^3\psi_{X^*}'(iy(x))+r_2(t,x), 
\end{equation}
where 
\begin{equation}\notag
|r_2(t,x)|\le 8\,C\,|t|^3|y(x)|T^{-2}+8C|t|^3(t^2+y(x)^2)T^{-3}\le 
16\,C(|t|+|y(x)|)|t|^3T^{-2} 
\end{equation}
for $|t|\le T/4$ and $|y(x)|\le T/4$. Thus, the lemma is proved.
\end{proof}

Our next step is to estimate the integrals
\begin{equation}\notag
I_p(x) \equiv \Re \int_{\mathbb R}(it)^p R(t,x,y(x))\,dt, \qquad p=0,1,2.
\end{equation}
To this aim, we need the following lemmas.

\begin{lemma}\label{lem7.3}
For $p=0,2$,
\begin{equation}
(-1)^{p/2}I_p(x) = \frac{p!}{2^{p/2}(p/2)!} \, 
\frac{\sqrt{2\pi}}{\gamma(y(x))^{(p+1)/2}} + r_p(x),
\quad |x|\le c''T_1/|b_1|,\notag
\end{equation}
where
\begin{equation}
|r_p(x)|\le c(|b_1|^{7/2}+|b_1|^{(p+3)/2}y(x)^2)T^{-2}\label{lem7.3.1'}
\end{equation}
with some absolute constants $c$.
\end{lemma}
\begin{proof}
For short we write $y$ in place
of $y(x)$. Put $T_2 = c''T/\sqrt{|b_1|}$ and write
\begin{equation}\notag
\int_{-\infty}^{\infty}
t^p\, \Re R(t,x,y)\,dt = I_{p1}+I_{p2} =
\Big(\int_{-T_2}^{T_2} +\int_{|t|\ge T_2}\Big)t^p\,\Re R(t,x,y)\,dt. 
\end{equation}
First consider the integral $I_{p1}$. We have
\begin{align}
I_{p1} = I_{p1,1} - I_{p1,2} & \equiv 
\int_{-T_2}^{T_2} t^p\,|R(t,x,y)|\,dt\notag\\
&-2\int_{-T_2}^{T_2}t^p|R(t,x,y)|\,
\sin^2\Big(\frac 12\,\Im\log R(t,x,y)\Big)\,dt. \notag
\end{align}
By (\ref{lem7.1.1}), we see that
\begin{equation}\notag
I_{p1,1}=\int_{-T_2}^{T_2} t^p\, e^{-\frac {\gamma(y)}{2}\,t^2}\,dt
+ \int_{-T_2}^{T_2} t^p\,
e^{-\frac {\gamma(y)}{2}\,t^2}\big(e^{r_1(t,x)}-1\big)\,dt.
\end{equation}
Using the inequality $|e^z-1|\le |z|e^{|z|},\,z\in\mathbb C$, and 
applying Lemma~\ref{lem7.1} together with (\ref{lem7.1.1''}), 
(\ref{lem7.1a.1}), (\ref{lem7.1a.2}), we have
\begin{align}\label{lem7.3.1}
&\Big|\int_{-T_2}^{T_2} t^p\,
e^{-\frac {\gamma(y)}{2}\,t^2} \big(e^{r_1(t,y)}-1\big)\,dt\Big| \le
\int_{-T_2}^{T_2} t^p\, 
e^{-\frac {\gamma(y)}{2}t^2}|r_1(t,y)|e^{|r_1(t,y)|}\,dt\notag\\
 & \le \int_{-T_2}^{T_2} t^p\,
 e^{-\frac {1}{4|b_1|}t^2}|r_1(t,y)|\,dt\le 
c\int_{-T_2}^{T_2}t^{p+2}\,
e^{-\frac {1}{4|b_1|}t^2}\frac{t^2+y^2}{T^2}\,dt\notag\\
&\le c|b_1|^{(p+3)/2}(|b_1|+y^2)T^{-2}.
\end{align}
On the other hand
\begin{equation}\label{lem7.3.2}
\int_{-T_2}^{T_2} t^p\,
e^{-\frac{\gamma(y)}{2}t^2}\,dt =
\frac{\sqrt{2\pi}\,p!}{2^{p/2}\,(p/2)!\,\gamma(y)^{(p+1)/2}} -
\int_{|t|\ge T_2} t^p\,e^{-\frac {\gamma(y)}{2}\,t^2}\,dt, 
\end{equation}
where, by (\ref{lem7.1a.1}) and by the assumption (\ref{2.0}) in the proof of 
Theorem~2.1 and by the assumption (\ref{8.0}) in the proof of Theorem~2.2,
\begin{equation}\label{lem7.3.3}
\int_{|t|\ge T_2} t^p\,e^{-\frac {\gamma(y)}{2}t^2}\,dt \le 
\frac{cT_2^{p-1}}{\gamma(y)}e^{-\frac 12(T_2\sqrt{\gamma(y)})^2}
\le c|b_1|^{(3-p)/2}T^{p-1}e^{-c''^2\frac {\gamma(y)}{2|b_1|}T^2}\le cT^{-4}.
\end{equation}
Therefore in view of (\ref{lem7.3.1})--(\ref{lem7.3.3}), we deduce
\begin{equation}\label{lem7.3.4}
I_{p1,1} = \frac{\sqrt{2\pi}\, p!}{2^{p/2}\,(p/2)!\gamma(y)^{(p+1)/2}} +
c\theta\frac{|b_1|^{(p+3)/2}(|b_1|+y^2)}{T^2}.
\end{equation}

Now let us turn to the integral $I_{p1,2}$. By (\ref{lem7.2.1}), 
we have
\begin{align}
|I_{p1,2}|&\le \frac 12\int_{-T_2}^{T_2}|
R(t,x,y)|\,(\Im\log R(t,x,y))^2\,dt\notag\\
&\le 2\int_{-T_2}^{T_2}
|R(t,x,y)|\,\big(t^6 |\psi_{X^*}'(iy)|^2+|r_2(t,x)|^2\big)\,dt.  \notag 
\end{align}
By Lemmas~\ref{lem7.1}--\ref{lem7.2} and by the assumptions 
(\ref{2.0}) and (\ref{8.0}), and by (\ref{lem7.1a.4}), we arrive at the upper bound
\begin{align}
|I_{p1,2}|\le \frac{c}{T^2} \int_{-\infty}^{\infty}
t^6\Big(\frac{t^2+y^2}{T^2}+1\Big) e^{-\frac 1{4|b_1|}t^2}\,dt 
\le\frac{c}{T^2}|b_1|^{7/2}\Big(\frac{|b_1|+y^2}{T^2}+1\Big)\le 
\frac{c|b_1|^{7/2}}{T^2}.\label{lem7.3.5}
\end{align}
It remains to estimate the integral $I_{p2}$. By (\ref{2.0}) in the proof of 
Theorem~2.1 and by (\ref{8.0}) in the proof of Theorem~2.2,
\begin{align}
|I_{p2}|&\le 2\int_{T_2}^{\infty}t^p|R(t,x,y)|\,dt\le 
2\int_{T_2}^{\infty}t^pe^{-\frac{\sigma^2}2t^2}\,dt\notag\\
&\le 2\int_{c''\sigma T}^{\infty}t^pe^{-\frac{\sigma^2}2t^2}\,dt\le 
c \sigma^{p-3}T^{p-1}e^{-(c'')^2\sigma^4T^2}
\le cT^{-4}.\label{lem7.3.6}
\end{align}
The assertion of the lemma follows from (\ref{lem7.3.4})--(\ref{lem7.3.6}).
\end{proof}

Let us return to the definition of the integrals $I_k(x,y(x))$, $k=0,1,2$, 
see (\ref{6.1*}). We note that $I_0(x,y(x))=\frac 1{2\pi}I_0(x)$ and,
by Lemma~\ref{lem7.3}, for $|x|\le c''T_1/|b_1|$,
\begin{equation}\label{7.1}
I_0(x,y(x))=\frac 1{\sqrt{2\pi\gamma(y(x))}}+\frac 1{2\pi}r_0(x),
\end{equation}
where $r_0(x)$ satisfies the inequality (\ref{lem7.3.1'}).

Since for $|x|>c''T_1/|b_1|$ we choose $y(x)=y(\pm c''T_1/|b_1|)$ and
since $|y(x)|\le c''T/|b_1|$ for such $x$, we obtain, using 
Lemmas~\ref{lem7.1} and \ref{lem7.1a}, and the assumptions (\ref{2.0}) or (\ref{8.0}), that
\begin{align}\label{7.2}
& |I_k(x,y(x))|\le 
\int_{|t|\le T_2}|t|^k|R(t,x,y(x))|\,dt +
\int_{|t|>T_2} |t|^k\,|R(t,x,y(x))|\,dt\notag\\
&\le 
\int_{-\infty}^\infty |t|^k\, e^{-\frac {t^2}{4|b_1|}}\,dt+
\int_{|t|>T_2}|t|^k\,e^{-\frac{\sigma^2t^2}2 }\,dt
\le c\Big(|b_1|^{\frac{k+1}2}+T_2^{k-1}\sigma^{-2}
e^{-\frac{\sigma^2T_2^2}{2}}\Big)\le c|b_1|^{\frac{k+1}2}
\end{align}
with some absolute constant $c$. The bound (\ref{7.2}) holds for 
$|x|\le c''T_1/|b_1|$ as well. Thus (\ref{7.2}) is valid for all real $x$.

The relations (\ref{7.1}) and (\ref{7.2}) allow us to control 
the behaviour of the integral $I_0(x,y(x))$.

\section{Estimation of the integrals $I_1(x,y(x))$ and $I_2(x,y(x))$}

In Section 8 we assume that (\ref{8.0}) holds.
As before we use the letter $c$ to denote absolute constants 
which may vary from place to place.

In order to get estimates on $I_1(x,y(x))$ and $I_2(x,y(x))$, which would
be  similar to (\ref{7.1}) and (\ref{7.2}), we need to bound the integral 
$I_1(x)=-\int_{\mathbb R} t\,\Im R(t,x,y(x))\,dt$.
Let us prove the following lemma.

\begin{lemma}\label{lem8.1}
For $|x|\le c''T_1/|b_1|$,
\begin{equation}\notag
-I_1(x) = 
3\sqrt{\frac{\pi}2}\,\gamma(y(x))^{-5/2}\, i\psi'_{X^*}(iy(x))+r_1(x),
\end{equation}
where
\begin{equation}\label{lem8.1.1}
|r_1(x)|\le 
c|b_1|^{7/2}(|b_1|^{1/2}+|y(x)|)\, T^{-2}
\end{equation}
with some absolute constant $c$.
\end{lemma}

\begin{proof}
Put $T_3 = c''T^{1/3}$ and rewrite $I_1(x)$ in the form
\begin{equation}\label{lem8.1.2}
-I_1(x)=I_{11}+I_{12} \equiv
\Big(\int_{-T_3}^{T_3} +
\int_{|t|>T_3}\Big)\, t\,\Im R(t,x,y(x))\,dt.
\end{equation}
Below in the proof of this lemma we simply write $y$ instead of $y(x)$. 
For $|t|\le T_3$, by Lemma~\ref{lem7.2},
\begin{align}
\sin(\Im \log R(t,x,y)) & =
\sin\Big(\frac i2t^3\psi'_{X^*}(iy)\Big)-
2\sin\Big(\frac i2t^3\psi'_{X^*}(iy)\Big)
\sin^2\Big(\frac 12r_2(t,x)\Big)\notag\\
& +
\cos\Big(\frac i2t^3\psi'_{x^*}(iy)\Big) \sin(r_2(t,x)).\notag
\end{align}
Using $|\sin x|\le |x|$ and $|\sin x-x|\le\frac 16|x|^3$ for 
$x\in\mathbb R$, we have
\begin{equation}\notag
\sin(\Im\log R(t,x,y))=\frac i2\,t^3\psi'_{X^*}(iy)+r_{1,1}(t,x),
\end{equation}
where
\begin{equation}\label{lem8.1.3*}
|r_{1,1}(t,x)| \, \le \, \frac{1}{48}\,|t|^9\,|\psi'_{X^*}(iy)|^3 +
|r_2(t,x)|\,\Big(\frac{1}{4}\,|r_2(t,x)||t^3\psi'_{X^*}(iy)|+1\Big).
\end{equation}
Therefore, one can rewrite the integral $I_{11}$ in the form
\begin{equation}\label{lem8.1.3}
I_{11} = \frac{i}{2}\,\psi'_{X^*}(iy) \int_{-T_3}^{T_3} t^4 |R(t,x,y)| \,dt+
\int_{-T_3}^{T_3} t |R(t,x,y)|\, r_{1,1}(t,x)\,dt,
\end{equation}
where the second integral on the right-hand side of (\ref{lem8.1.3}) 
does not exceed, by (\ref{lem8.1.3*}) and by (\ref{lem7.1a.4}) and (\ref{lem7.1.2}),
the quantity
\begin{equation}\notag
c\int_{-T_3}^{T_3}
\bigg[\frac{t^{10}}{T^3}+\frac{(|t|+|y|)t^4}{T^2}
\Big(\frac{(|t|+|y|)|t|^3}{T^3}+1\Big)\bigg]\,|R(t,x,y)|\,dt
\end{equation}
with some absolute constant $c$. We see, by Lemmas~\ref{lem7.1} and \ref{lem7.1a} 
and by (\ref{8.0}), that 
\begin{equation}
|R(t,x,y)|\le e^{-(\frac 3{8|b_1|}+\frac 1{8|b_1|})t^2}=
e^{-\frac 1{4|b_1|}t^2}\quad \text{for}\quad |t|\le T_3\le c''T/\sqrt{|b_1|},
\quad |y|\le c''T/|b_1|,
\end{equation} 
so, using again (\ref{8.0}),
the above integral does not exceed
\begin{equation}\label{lem8.1.4a}
\frac{c}{T^2}
\Big(|b_1|^{5/2}(|b_1|^{1/2}+|y|)+
\frac{|b_1|^{11/2}}T+\frac{b_1^4(|b_1|+y^2)}{T^3}\Big)\le 
\frac{c}{T^2}|b_1|^{7/2}(|b_1|^{1/2}+|y|).
\end{equation}
Repeating the arguments which 
we used in the proof of (\ref{lem7.3.4}), we easily obtain that
the first summand on the right-hand side of (\ref{lem8.1.3}) is equal to
\begin{align}\label{lem8.1.4}
&3\sqrt{\frac{\pi}2}\gamma(y)^{-5/2}i\psi'_{X^*}(iy)+
c\theta |\psi'_{X^*}(iy)|\frac{|b_1|^{7/2}(|b_1|+y^2)}{T^2}\notag\\
&=3\sqrt{\frac{\pi}2}\gamma(y)^{-5/2}i\psi'_{X^*}(iy)+
c\theta\frac{|b_1|^{7/2}(|b_1|+y^2)}{T^3}.
\end{align}
It remains to estimate the integral 
$I_{12}$ similarly as in the proof of
(\ref{lem7.3.6}). Namely, using the assumption (\ref{8.0}), we obtain that 
\begin{equation}\label{lem8.1.6}
|I_{12}|\le 2\int_{T_3}^{\infty}t|R(t,x,y)|\,dt\le 
2\int_{T_3}^{\infty}te^{-\frac 12\sigma^2 t^2}\,dt=
\frac 2{\sigma^2}e^{-\frac 12\sigma^2 T_3^2}
=\frac 2{\sigma^2}e^{-\frac 12(\sigma c'')^2 T^{2/3}}\le\frac 1{T^4}.
\end{equation}
Taking into account (\ref{8.0}), we see that the assertion of the lemma 
follows immediately from (\ref{lem8.1.4a})--(\ref{lem8.1.6}).
\end{proof}

Recalling the definition of the integrals $I_1(x,y(x))$ and $I_2(x,y(x))$, 
we see that
\begin{align}
I_1(x,y(x)) & =\frac 1{2\pi}\,\big(y(x)I_0(x)-I_1(x)\big),\notag\\
I_2(x,y(x)) = & \frac 1{2\pi}\,
\big(y(x)^2I_0(x)-2y(x)I_1(x)+I_2(x)\big).\label{8.1}
\end{align}
Using Lemma~\ref{lem7.3} and Lemma~\ref{lem8.1}, we conclude from that, 
for $|x|\le c''T_1/|b_1|$,
\begin{align}\label{8.2}
I_2(x,y(x)) &=
\frac 1{\sqrt{2\pi\gamma(y(x))}}\,
\Big(y(x)^2+\frac{3y(x)}{\gamma(y(x))^2}\, i\psi'_{X^*}(iy(x))
-\frac 1{\gamma(y(x))}\Big)\notag\\
&+\frac 1{2\pi}\,(r_0(x)y^2(x)+2r_1(x)y(x)-r_2(x)),
\end{align}
where $r_0(x),r_2(x)$ and $r_1(x)$ admit the upper bounds 
(\ref{lem7.3.1'}), (\ref{lem8.1.1}), respectively.

\section{End of the proof of Theorem~1.1}
Starting from the hypothesis (2.1), we need to derive
a good upper bound for $D(X_{\sigma})$, which is equivalent to bounding
the relative entropy $D(X^*_{\sigma})$, according to Lemma~\ref{lem4}.
This will be done with the help of the relations (\ref{6.12}), (\ref{6.13}) 
and (\ref{7.1}), (\ref{7.2}) for the density $p_{X^*_{\sigma}}(x)$ of the 
random variable $X^*_{\sigma}$. First, let us prove the following lemma.

\begin{lemma}\label{lem9.1}
For $|x|\le c''T_1/|b_1|$,
\begin{equation}\notag
\log\frac{p_{X^*_{\sigma}}(x)}{\varphi_{\sqrt{1/|b_1|}}(x)} =
\frac{c_3}{2T}\, \Big((b_1x)^3+3b_1y(x)\Big) + \tilde{r}(x),
\end{equation}
where with some absolute constant $c$
\begin{equation}\label{lem9.1.1a}
|\tilde{r}(x)|\le\frac c{T^2}\big(b_1^2y(x)^2+|b_1|^{3}+|b_1|^5x^4\big).
\end{equation}
\end{lemma}

\begin{proof}
By (\ref{6.12}) and (\ref{7.1}), we have, for $|x|\le c''T_1/|b_1|$,
\begin{equation}\label{lem9.1.1}
\log\frac{p_{X^*_{\sigma}}(x)}{\varphi_{1/\sqrt{|b_1|}}(x)} =
\frac 12\,c_3b_1^3\, \frac{x^3}T + \frac{c\theta b_1^5}{T^2}\,x^4
-\frac 12\log(|b_1|\gamma(y(x))+
\log\Big(1+\sqrt{\frac {\gamma(y(x))}{2\pi}}r_0(x)\Big).
\end{equation}
Recalling (\ref{lem7.1.1'}) and (\ref{lem5d.7}), we see that
\begin{equation}\label{lem9.1.2}
|b_1|\gamma(y(x))=1+|b_1|(\psi_{X^*}(iy(x))+2iy(x)\psi'_{X^*}(iy(x)))=
1+3c_3|b_1|y(x)T^{-1}+\rho_1(x),
\end{equation}
where 
\begin{equation}\notag
\rho_1(x) \equiv |b_1|\sum_{k=4}^{\infty}
i^k(2k-3)\,c_k\Big(\frac{iy(x)}T\Big)^{k-2}.
\end{equation}
It is easy to see that
\begin{equation}\label{lem9.1.1'}
|\rho_1(x)|\le 8 C|b_1|\Big(\frac{y(x)}T\Big)^2\le \frac 14.
\end{equation}
Since $\frac{|3c_3\,b_1y(x)|}{T} \le \frac{1}{4}$, and using 
$|\log(1+u)-u|\le u^2$ ($|u|\le 1/2$), we get from (\ref{lem9.1.2}) that
\begin{equation}\label{lem9.1.3}
\log(|b_1|\,\gamma(y(x))) =
\frac{3c_3|b_1|y(x)}T+c\theta\Big(\frac{b_1y(x)}T\Big)^2
\end{equation}
with some absolute constant $c$. Now we conclude from  (\ref{lem7.1a.1}) 
and (\ref{lem7.3.1'}) that
\begin{equation}\notag
\sqrt{\frac{\gamma(y(x))}{2\pi}}|r_0(x)|\le 
c\,|b_1|\frac{b_1^{2}+y(x)^2}{T^2}\le \frac 14
\end{equation}
and arrive as before at the upper bound
\begin{equation}\label{lem9.1.4}
\big|\log\Big(1+\sqrt{\frac{\gamma(y(x))}{2\pi}} r_0(x)\Big)\big|\le 
c\,|b_1|\frac {b_1^{2}+y(x)^2}{T^2}.
\end{equation}
Applying (\ref{lem9.1.3}) and (\ref{lem9.1.4}) to (\ref{lem9.1.1}), 
we obtain the assertion of the lemma.
\end{proof}

To estimate the quantity $D(X^*_{\sigma})$, we represent it as
\begin{equation}\label{9.1aa}
D(X^*_{\sigma}) = J_1+J_2 =
\Big(\int_{-c''T_1/|b_1|}^{c''T_1/|b_1|} +
\int_{|x|>c''T_1/|b_1|}\Big)p_{X^*_{\sigma}}(x)
\log\frac{p_{X^*_{\sigma}}(x)}{\varphi_{\sqrt{1/|b_1|}}(x)}\,dx.
\end{equation}

First let us estimate $J_1$. 
using the letters $c,C'$ to denote 
absolute positive constants which may vary from place to place. 
By Lemma~\ref{lem9.1},
\begin{align}
J_1 = \frac {c_3}TJ_{1,1}+J_{1,2} =
\frac {c_3}{2T}\int_{-\frac{c''T_1}{|b_1|}}^{\frac{c''T_1}{|b_1|}}\,
p_{X^*_{\sigma}}(x)\Big((b_1x)^3 + 3b_1y(x)\Big)dx
+\int_{-\frac{c''T_1}{|b_1|}}^{\frac{c''T_1}{|b_1|}} 
p_{X^*_{\sigma}}(x)\tilde{r}(x)\,dx.\label{9.1a}
\end{align}
Using (\ref{6.12}) and (\ref{7.1}), we note that
\begin{align}
&\int_{-c''T_1/|b_1|}^{c''T_1/|b_1|}x^3p_{X^*_{\sigma}}(x)\,dx =
\int_{-c''T_1/|b_1|}^{c''T_1/|b_1|}x^3(p_{X^*_{\sigma}}(x)
- \varphi_{\sqrt{1/|b_1|}}(x))\,dx =
J_{1,1,1}+J_{1,1,2}+J_{1,1,3}\notag\\
& =
\int_{-c''T_1/|b_1|}^{c''T_1/|b_1|}x^3\varphi_{\sqrt{1/|b_1|}}(x)
\Big(\frac 1{\sqrt{\gamma(y(x))|b_1|}}-1\Big)
e^{c_3b_1^3x^3/(2T)+c\theta b_1^5x^4/T^2}\,dx\notag\\
 &+
\int_{-c''T_1/|b_1|}^{c''T_1/|b_1|}x^3\varphi_{\sqrt{1/|b_1|}}(x)
\Big(e^{c_3b_1^3x^3/(2T)+c\theta b_1^5x^4/T^2}-1\Big)\,dx\notag\\
 &+
\frac 1{\sqrt{2\pi |b_1|}}\int_{-c'T_1/|b_1|}^{c''T_1/|b_1|}
x^3\varphi_{\sqrt{1/|b_1|}}(x) 
e^{c_3b_1^3x^3/(2T)+c\theta b_1^5x^4/T^2}r_0(x)\,dx.
\notag
\end{align}
It is easy to see that 
\begin{equation}\label{9.1b}
\frac{|c_3|\,|b_1|^3|x|^3}{2T} + \frac{c|b_1|^5 x^4}{T^2} \le 
\frac{|b_1|x^2}{4} \quad \text{for}\quad |x|\le \frac{c''T_1}{|b_1|}. 
\end{equation}
With the help of (\ref{lem9.1.2})--(\ref{lem9.1.1'}) and 
using the bound $|(1+u)^{-1/2}-1|\le |u|,\,|u| \le \frac{1}{2}$, we get
\begin{equation}\notag
|\,(\gamma(x)|b_1|)^{-1/2}-1|\le c\,\frac{|b_1|\,|y(x)|}{T}.
\end{equation}
The last estimates and (\ref{lem6.11'}) lead to
\begin{equation}\label{9.1}
|J_{1,1,1}|\le \frac{c|b_1|}T
\int_{-\frac{c''T_1}{|b_1|}}^{\frac{c''T_1}{|b_1|}} 
|x|^3 |y(x)|\sqrt{|b_1|}\, e^{-|b_1|x^2/4}\,dx
\le \frac{cb_1^{5/2}}T\int_{-\infty}^\infty x^4e^{-|b_1|x^2/4}\,dx\le 
\frac cT.
\end{equation}
Applying $|e^u-1|\le |u|e^{|u|}$, we have, for $|x|\le c'T_1/|b_1|$,
\begin{equation}\notag
\big|e^{c_3b_1^3x^3/(2T)+c\theta b_1^5x^4/T^2}-1\big|\le
c|b_1|^3|x|^3\Big(\frac{1}{2T}+\frac{b_1^2|x|}{T^2}\Big)\,e^{|b_1|x^2/4}.
\end{equation}
Therefore, we deduce the estimate
\begin{equation}\label{9.2}
|J_{1,1,2}|\le c|b_1|^{7/2}\int_{-\infty}^\infty x^6
\Big(\frac{1}T + \frac {b_1^2|x|}{T^2}\Big)\,e^{-|b_1|x^2/4}\,dx
\le c \Big(\frac{1}{T} + \frac{|b_1|^{3/2}}{T^2}\Big).
\end{equation}
By (\ref{lem6.11'}) and (\ref{lem7.3.1'}), we immediately get
\begin{equation}\label{9.3}
|J_{1,1,3}| 
 \le 
c\,T^{-2} \int_{-c'' T_1/|b_1|}^{c'' T_1/|b_1|}
|x|^3\,\big(|b_1|^{7/2} + |b_1|^{3/2}\, y(x)^2\big)\, e^{-|b_1|\,x^2/4}\,dx 
 \le 
c |b_1|^{3/2}\, T^{-2}.
\end{equation} 
Hence, by (\ref{9.1})--(\ref{9.3}) and (\ref{2.0}),
\begin{equation}\label{9.3a}
\Big|\int_{-c''T_1/|b_1|}^{c''T_1/|b_1|} x^3p_{X^*_{\sigma}}(x)\,dx\Big| 
\le 
c\, \Big(\frac {1}{T}+\frac {|b_1|^{3/2}}{T^2}\Big)\le \frac{c}{T}.
\end{equation} 
In the same way,
\begin{equation}\label{9.3b}
\Big|\int_{-c''T_1/|b_1|}^{c''T_1/|b_1|}x p_{X^*_{\sigma}}(x)\,dx\Big| \le 
c\,\Big(\frac {|b_1|}{T}+\frac {|b_1|^{5/2}}{T^2}\Big)\le \frac{c|b_1|}{T}.
\end{equation}

Recalling (\ref{6.10}), we see that $y(x)=b_1x+c\theta b_1^2x^2/T_1$. 
As a result, using (\ref{9.3a})--(\ref{9.3b}) and the property 
$\Var(X)\le 1$, we come to the upper bound
\begin{equation}\label{9.4a}
|J_{1,1}|\le c\,|b_1|^3 T^{-1}.
\end{equation}

In order to estimate $J_{1,2}$, we employ the inequality (\ref{lem9.1.1a}).
Recalling (\ref{6.12}), (\ref{7.2}) and (\ref{9.1b}), we then have
\begin{equation}\label{9.4}
|J_{1,2}| \le \frac c{T^2}\int_{-c'T_1/|b_1|}^{c'T_1/|b_1|}
\big(b_1^2y(x)^2 + |b_1|^{3} + |b_1|^5 x^4\big)\sqrt{|b_1|}\,
e^{-|b_1|x^2\/4}\,dx
\le c |b_1|^3\, T^{-2}.
\end{equation}
Combining (\ref{9.4a}) and (\ref{9.4}), we arrive at
\begin{equation}\label{9.5}
|J_1|\le c |b_1|^3 T^{-2}.
\end{equation}

Let us estimate $J_2$. From (\ref{6.13}), \ref{7.2}),
we have, for all $|x|>c''T_1/|b_1|$,
\begin{equation}\label{9.6}
p_{X^*_{\sigma}}(x) \le 
C'\sqrt{|b_1|}\, e^{-cT|x|/|b_1|}\le C'\sqrt{|b_1|}\, 
e^{-cc'c''T^2/|b_1|^3}<1.
\end{equation}
Here we also used (\ref{2.0}) and the assumption that 
$0 < \ep \le \ep_0$, where $\ep_0$ is a sufficiently small absolute 
constant. Using (\ref{9.6}) and (\ref{2.0}), we easily obtain
\begin{align}\label{9.7}
J_2&\le -\int_{|x|>c''T_1/|b_1|}p_{X^*_{\sigma}}(x)\log\varphi_{\sqrt {1/|b_1|}}(x)\,dx
=\frac 12\log\frac{2\pi}{|b_1|}
\int_{|x|>c''T_1/|b_1|}p_{X^*_{\sigma}}(x)\,dx\notag\\
&+\frac {|b_1|}2\int_{|x|>c''T_1/|b_1|}x^2p_{X^*_{\sigma}}(x)\,dx
\le C'\sqrt{|b_1|}\int_{|x|>c''T_1/|b_1|}
\frac 12\,(\log(4\pi)+|b_1|x^2)e^{-cT|x|/|b_1|}\,dx\notag\\
&\le C'\,(|b_1|^{3/2}T^{-1}+|b_1|^{3/2}T)\,
e^{-cc'c''T^2/|b_1|^3}\le C'T^{-2}.
\end{align}
Thus, we derive from (\ref{9.5}) and (\ref{9.7}) the inequality
$D(X^*_{\sigma})\le c|b_1|^3T^{-2}$.
Recalling (\ref{lem4.a}) and Lemma~\ref {lem1}, we finally conclude, using (\ref{2.0}), that
\begin{equation}\label{9.8}
D(X_{\sigma})\le 
c \frac{|b_1|^3}{T^2}+c\Big(\frac N {\sigma}\Big)^{3}\sqrt{\varepsilon} \le 
\frac {c}{(v_1^2+\sigma^2)^3 T^2} + 
c \Big(\frac N{\sigma}\Big)^{3}
\sqrt{\varepsilon}\le \frac{c}{(v_1^2+\sigma^2)^3 T^2}.
\end{equation}

An analogous inequality also holds for the random variable $Y_{\sigma}$, 
and so Theorem~1.1 follows from these estimates.

\begin{remark}
Under the assumptions of Theorem~1.1, a stronger inequality than 
(\ref{l1}) follows from (\ref{9.8}), namely 
\begin{equation}
D(X_{\sigma} + Y_{\sigma}) \, \ge \, e^{c\sigma^{-6}\log \sigma} 
\Big[\exp\Big\{-\frac{c}{(\Var(X_{\sigma}))^3\, D(X_{\sigma})}\Big\} +
\exp\Big\{-\frac{c}{(\Var(Y_{\sigma}))^3\, D(Y_{\sigma})}\Big\}\Big] \notag
\end{equation}
with some positive absolute constant $c$.
\end{remark}

\section{Proof of Theorem 1.2}

The proof of Theorem~1.2 is very similar to the proof of Theorem~1.1. 
Therefore, keeping the same notations, we omit some routine calculations. 

Let $J_{st}(X_{\sigma}+Y_{\sigma}) \le \ep < 1$. Then, by (\ref{l4}),
$D(X_{\sigma}+Y_{\sigma}) \le \frac{1}{2}\,\ep$, so that
one can use the previous arguments of the proof of Theorem 1.1. 
Without loss of generality we assume that (\ref{8.0}) holds.
As before, we use the letter $c$ to denote 
absolute constants which may vary from place to place. 

To derive an upper bound for the standardized Fisher information
$J_{st}(X_{\sigma}^*)$, we represent it as
\begin{equation}\label{10.1}
\frac{J_{st}(X_{\sigma}^*)}{(\sigma_1^2 + \sigma^2)}
= J_1 + J_2 =
\bigg(-\int_{-\frac{c''T_1}{|b_1|}}^{\frac{c''T_1}{|b_1|}} -
\int_{|x| > \frac{c''T_1}{|b_1|}}\bigg)\, 
p_{X_{\sigma}^*}''(x)
\log\frac{p_{X^*_\sigma}(x)}{\varphi_{\sqrt{1/|b_1|}}(x)}\,dx,
\end{equation}
which is analogous to (\ref{9.1aa}).

Now, using (\ref{lem7.1a.4}), (\ref{lem7.3.1'}) and (\ref{lem8.1.1}), 
rewrite (\ref{8.2}) in the form, for $|x|\le c''T_1/|b_1|$,
\begin{align}
I_2(x,y(x)) & = Q_1(x) + \frac{c\theta}T\, Q_2(x) + 
\frac{c\theta}{T^2}\,Q_3(x) \equiv
\frac{1}{\sqrt{2\pi\gamma(y(x))}}\, \big(y(x)^2-\gamma(y(x))^{-1}\big)
\notag\\
& + \frac{c\theta}T\frac{y(x)}{\gamma(y(x))^{5/2}} + 
\frac {c\theta}{T^2}\big(|b_1|^{7/2}(y(x)^2+1 +|b_1|^{1/2}|y(x)|)+
|b_1|^{3/2}y(x)^4\big).
\label{10.2}
\end{align}
Recalling (\ref{6.10}), we see that 
$y(x) = b_1x + \frac{c\theta b_1^3x^2}{T}$ and 
$y(x)^2 = b_1^2 x^2 + \frac{c\theta b_1^4x^3}{T}$.
Therefore from these relations and from (\ref{lem9.1.2})--(\ref{lem9.1.1'}), 
we deduce the representation
\begin{equation}\label{10.3}
Q_1(x) = \sqrt{\frac{|b_1|}{2\pi}}\,|b_1|(|b_1|\,x^2-1) +
\frac{c\theta}{T}\, |b_1|^{7/2}|x|(|b_1|x^2+1).
\end{equation}

By Lemma~\ref{lem9.1}, we obtain an analog of (\ref{9.1a}),
\bee
J_1 
 & = &
 -\frac{c_3}{T}\, J_{1,1} - J_{1,2} \\
 & = &
\frac{c_3}{2T}\int_{-\frac{c''T_1}{|b_1|}}^{\frac{c''T_1}{|b_1|}}
p''_{X_{\sigma}^*}(x)\Big((b_1x)^3 +3b_1 y(x)\Big)\,dx +
\int_{-\frac{c''T_1}{|b_1|}}^{\frac{c''T_1}{|b_1|}} 
p''_{X_{\sigma}^*}(x)\tilde{r}(x)\,dx.
\ene
Put
\bee
J_{1,1}
 & = & J_{1,1,1} + J_{1,1,2} \\ 
 & = &
-\frac 12\int_{-\frac{c''T_1}{|b_1|}}^{\frac{c''T_1}{|b_1|}} 
p''_{X_{\sigma}^*}(x) \Big((b_1x)^3+3b_1^2x\Big)\,dx -
\frac{cb_1^4}{T}
\int_{-\frac{c''T_1}{|b_1|}}^{\frac{c''T_1}{|b_1|}} 
p''_{X_{\sigma}^*}(x)\,\theta x^2\,dx.
\ene
Taking into account (\ref{6.12}), (\ref{7.2}) and (\ref{9.1b}), we obtain
\begin{equation}\label{10.4}
|J_{1,1,2}|\le 
c\frac{|b_1|^{11/2}}T \int_{-c''T_1/|b_1|}^{c''T_1/|b_1|}x^2 e^{-|b_1|x^2/4}\,dx\le 
c\frac{b_1^4}T.
\end{equation}
Consider the integral
\begin{equation}\label{10.5}
\tilde{J}_1 =
\frac 12\int_{-\frac{c''T_1}{|b_1|}}^{\frac{c''T_1}{|b_1|}} 
|b_1|(|b_1|x^2-1)
((b_1x)^3+3b_1^2x)\varphi_{\sqrt{1/|b_1|}}(x)\,
e^{\frac{c_3b_1^3x^3}{2T} + \frac{c\theta b_1^5x^4}{T^2}}\,dx. 
\end{equation}
Using (\ref{8.0}), we obtain, for $k=0,1,\dots$, 
\begin{align}
&
\Big|\int_{-c''T_1/|b_1|}^{c''T_1/|b_1|} 
x^{2k+1}\varphi_{\sqrt{1/|b_1|}}(x)\, 
e^{c_3b_1^3x^3/(2T)+c\theta b_1^5x^4/T^2}\,dx\Big|\notag\\
& =
\Big|\int_{-c''T_1/|b_1|}^{c''T_1/|b_1|} 
x^{2k+1} \varphi_{\sqrt{1/|b_1|}}(x)\,
(e^{c_3b_1^3x^3/(2T)+c\theta b_1^5x^4/T^2}-1)\,dx\Big|\notag\\
& \le 
c|b_1|^{1/2}\int_{-c''T_1/|b_1|}^{c''T_1/|b_1|}
|x|^{2k+1}\big(|b_1|^3|x|^3T^{-1}+|b_1|^5x^4T^{-2}\big)\,e^{-|b_1|x^2/4}\,dx
\notag\\
&\le 
c(k)|b_1|^{1-k}\big(T^{-1}+|b_1|^{3/2}T^{-2}\big)\le c(k)|b_1|^{1-k}T^{-1}\notag
\end{align}
with some constants $c(k)$ depending on $k$ only. We easily conclude from 
this estimate that
\begin{equation}\label{10.6}
|\tilde{J}_1|\le c b_1^4\, T^{-1}.
\end{equation}
For the integral
\begin{equation}
\tilde{J}_2 =
\int_{-\frac{c''T_1}{|b_1|}}^{\frac{c''T_1}{|b_1|}}
\frac{\theta}{T}\, |b_1|^3 |x| (|b_1|x^2+1)
((b_1x)^3+3b_1^2x)
\varphi_{\sqrt{1/|b_1|}}(x)\,
e^{\frac{c_3b_1^3x^3}{2T} + \frac{c\theta b_1^5 x^4}{T^2}}\,dx \notag
\end{equation}
we have a straightforward upper bound
\begin{equation}\label{10.7}
|\tilde{J}_2|\le \frac {cb_1^{11/2}}T
\int_{-\frac{c''T_1}{|b_1|}}^{\frac{c''T_1}{|b_1|}}
x^2(|b_1|x^2+1)^2\, e^{-|b_1|x^2/4}\,dx \le  \frac{cb_1^4}{T}.
\end{equation}
From (\ref{10.6}), (\ref{10.7}) we deduce the estimate
\begin{equation}
\frac 1{\sqrt{|b_1|}}\Big|\int_{-\frac{c''T_1}{|b_1|}}^{\frac{c''T_1}{|b_1|}}
Q_1(x)\big((b_1x)^3 +3 b_1^2x\big)
\varphi_{\sqrt{1/|b_1|}}(x)\,
e^{\frac{c_3b_1^3x^3}{2T} + \frac{c\theta b_1^5x^4}{T^2}}\,dx\Big|
\le \frac{cb_1^4}{T}. \label{10.8}
\end{equation}
In the same way,
\begin{equation}\label{10.9}
\frac 1{\sqrt{|b_1|}}\Big|\int_{-\frac{c''T_1}{|b_1|}}^{\frac{c''T_1}{|b_1|}}
Q_2(x)((b_1x)^3+3b_1^2x) \varphi_{\sqrt{1/|b_1|}}(x)\,
e^{\frac{c_3b_1^3x^3}{2T} + \frac{c\theta b_1^5x^4}{T^2}}\,dx\Big|\le 
\frac {cb_1^4}{T}
\end{equation}
and, taking into account (\ref{8.0}),
\begin{equation}\label{10.10}
\frac 1{\sqrt{|b_1|}}\Big|\int_{-\frac{c''T_1}{|b_1|}}^{\frac{c''T_1}{|b_1|}}
Q_3(x)((b_1x)^3+3b_1^2x) \varphi_{\sqrt{1/|b_1|}}(x)\,
e^{\frac{c_3b_1^3x^3}{2T} + \frac{c\theta b_1^5x^4}{T^2}}\,dx\Big|
\le \frac {cb_1^{11/2}}{T^2}\le\frac {cb_1^4}{T}.
\end{equation}
The bounds (\ref{10.8})--(\ref{10.10}) yield $|J_{1,1,1}|\le cb_1^4T^{-1}$, and taking
into account (\ref{10.4}) we have
\begin{equation}\label{10.11}
|J_{1,1}|\le cb_1^4T^{-1}.
\end{equation}
In view of (\ref{6.12}), (\ref{7.2}), (\ref{lem9.1.1a}) 
and (\ref{9.1b}), we get
\begin{equation}\label{10.12}
|J_{1,2}|\le \frac{c|b_1|^{9/2}}{T^2}
\int_{-c''T_1/|b_1|}^{c''T_1/|b_1|}(|b_1|x^2+1+b_1^2x^4)\,e^{-|b_1|x^2/4}\,dx
\le cb_1^4\,T^{-2}. 
\end{equation}

To estimate the integral $J_2$, we use (\ref{lem4.1}), (\ref{6.13}) and
(\ref{7.2}). We have, taking into account (\ref{8.0}),
\begin{align}\label{10.13}
|J_2|&\le 
c|b_1|^{3/2}\int_{|x|>c''T_1/|b_1|}
\Big(\frac{x^2+1}{\sigma^2}+|b_1|x^2+|\log |b_1||\Big)\,  e^{-cT|x|/|b_1|}\,dx\notag\\
 & \le 
c|b_1|^{3/2}\Big(\sigma^{-2}
\Big(\frac T{|b_1|^3}+\frac {|b_1|}T\Big)+\frac T{b_1^2}+
\frac{|b_1||\log |b_1||}{T}\Big)\,
e^{-cc'c''T^2/|b_1|^3}\le cT^{-2}.
\end{align}
Applying (\ref{10.11})--(\ref{10.13}) to (\ref{10.1}), we get
\begin{equation}\notag
J_{st}(X^*_{\sigma})\le  c|b_1|^3T^{-2}.
\end{equation}
Now, by Lemmas~\ref{lem1}, \ref{lem5c} and by (\ref{8.0}), the above 
inequality gives
\begin{equation}\label{10.14}
J_{st}(X_{\sigma})
\le \Big(1+\frac{v_1^2-\sigma_1^2}{\sigma_1^2+\sigma^2}\Big) 
J_{st}(X^*_{\sigma})+c\sigma^{-7}N^3\sqrt{\varepsilon}
\le \frac{C'}{(v_1^2+\sigma^2)^3T^{2}} 
\end{equation}
with some positive absolute constants $c,C'$. An analogous inequality also
holds for the random variable $Y_{\sigma}$, and thus Theorem 1.2 is proved.

\begin{remark}
Under the assumptions of Theorem~1.2, a stronger inequality than (\ref{l6}) 
follows from (\ref{10.14}),
\begin{equation}
J_{st}(X_{\sigma} + Y_{\sigma}) \ge e^{c\sigma^{-6}(\log \sigma)^3}
\Big[\exp\Big\{-\frac{c}{(\Var(X_{\sigma}))^3 J_{st}(X_{\sigma})}\Big\}
+ \exp\Big\{-\frac{c}{(\Var(Y_{\sigma}))^3J_{st}(Y_{\sigma})}\Big\}
\Big]\notag
\end{equation}
with some positive absolute constant $c$.
\end{remark}

\section{Proof of Theorem~1.3}

In order to construct random variables $X$ and $Y$ with the desired 
properties, we need some auxiliary results. We use the letters 
$c,c',\tilde{c}$ (with indices or without) 
to denote absolute positive constants which may vary from place to place, 
and $\theta$ may be any number such that $|\theta| \leq 1$.
First we analyze the function $v_\sigma$ with Fourier transform
$$
f_{\sigma}(t)=\exp\{-(1+\sigma^2)\,t^2/2+it^3/T\}, \qquad t\in\mathbb R
$$ 

\begin{lemma}\label{lem11.1}
If the parameter $T>1$ is sufficiently large and $0\le\sigma\le 2$,
the function $f_{\sigma}$ admits the representation
\begin{equation}\label{lem11.1.1}
f_{\sigma}(t) = \int_{-\infty}^\infty e^{itx} v_{\sigma}(x)\,dx
\end{equation}
with a real-valued infinitely differentiable function $v_{\sigma}(x)$ 
which together with its all derivatives is integrable and satisfies
\begin{align}
v_{\sigma}(x)>0,\qquad\qquad\text{for}&\quad x\le 
(1+\sigma^2)^2\,T/16; \label{lem11.1.2}\\
 |v_{\sigma}(x)|\le e^{-(1+\sigma^2)Tx/32},\quad\text{for}&\quad x\ge 
(1+\sigma^2)^2\,T/16.\label{lem11.1.2'}
\end{align}
In addition, for $|x|\le (1+\sigma^2)^2\,T/16$,
\begin{equation}\label{lem11.1.2''}
C_1 e^{-2(5-\sqrt 7)\,|xy(x)|/4}\le |v_{\sigma}(x)|\le C_2e^{-4|xy(x)|/9},
\end{equation}
where
\begin{equation}\label{lem11.1.3a}
y(x) = \frac{1}{6}\, T\,\big(-(1+\sigma^2)+\sqrt{(1+\sigma^2)^2-12x/T}\,\big).
\end{equation}
The right-hand side of the inequality $(\ref{lem11.1.2''})$ continues to 
hold for all $x\le (1+\sigma^2)^2T/16$.
\end{lemma}
\begin{proof}
Since $f_{\sigma}(t)$ decays very fast at infinity, the function 
$v_{\sigma}$ is given according to the inversion formula by
\begin{equation}\label{lem11.1.3}
v_{\sigma}(x) = \frac{1}{2\pi}\int_{-\infty}^\infty
e^{-ixt}\,f_{\sigma}(t)\,dt,\quad x\in\mathbb R.
\end{equation}
Clearly, it is infinitely many times differentiable and all its 
derivatives are integrable. It remains to prove
(\ref{lem11.1.2})--(\ref{lem11.1.2''}). By the Cauchy theorem, one may
also write
\begin{equation}\label{lem11.1.4}
v_{\sigma}(x) = e^{yx} f_{\sigma}(iy)\,\frac{1}{2\pi}
\int_{-\infty}^\infty e^{-ixt} R_{\sigma}(t,y)\,dt, \quad {\rm where} \ \ 
R_{\sigma}(t,y) = \frac{f_{\sigma}(t+iy)}{f_{\sigma}(iy)},
\end{equation}
for every fixed real $y$. Here we choose $y=y(x)$ according to the 
equality in (\ref{lem11.1.3a}) for $x\le (1+\sigma^2)^2\,T/16$.
In this case, it is easy to see,
\begin{equation}\notag
e^{-ixt}R_{\sigma}(t,y(x)) =
\exp\Big\{-\frac{(1+\sigma^2)t^2}2\Big(1+\frac{6y(x)}{(1+\sigma^2)T}\Big) 
+ i\frac{t^3}T\Big\} \equiv 
\exp\Big\{-\frac{\alpha(x)}{2}\,t^2+i\frac{t^3}T\Big\}.
\end{equation}
Note that $\alpha(x)\ge (1+\sigma^2)/2$ for $x$ as above.

For a better understanding of the behaviour of the integral in the 
right-hand side of (\ref{lem11.1.4}), put
$$
\tilde{I} = \frac{1}{2\pi}
\int_{-\infty}^\infty e^{-ixt} R_{\sigma}(t,y)\,dt
$$
and rewrite it in the form
\begin{equation}\label{lem11.1.5}
\tilde{I}=\tilde{I}_1+\tilde{I}_2 = 
\frac{1}{2\pi} \bigg(\int_{|t|\le T^{1/3}} +
\int_{|t|>T^{1/3}}\bigg)\,e^{-ixt}R_{\sigma}(t,y(x))\,dt.
\end{equation}
Using $|\cos u-1+u^2/2|\le u^4/4!$ ($u\in\mathbb R$), 
we easily obtain the representation
\begin{align}
\tilde{I}_1&=\frac 1{2\pi}\int_{|t|\le T^{1/3}} 
\Big(1-\frac{t^6}{2T^2}\Big)\,e^{-\alpha(x)t^2/2}\,dt +
\frac{\theta}{4!\,T^4}\,\frac 1{2\pi}
\int_{|t|\le T^{1/3}}t^{12}e^{-\alpha(x)t^2/2}\,dt\notag\\
& = \frac 1{\sqrt{2\pi\alpha(x)}}\,
\Big(1 - \frac{15}{2\alpha(x)^3\,T^2} + \frac{c\theta}{\alpha(x)^6 T^4}\Big)
-\frac 1{2\pi}\int_{|t|>T^{1/3}}\Big(1-\frac{t^6}{2T^2}\Big)\,
e^{-\alpha(x)t^2/2}\,dt.
\end{align}
The absolute value of last integral does not exceed 
$c(T^{1/3}\alpha(x))^{-1}e^{-\alpha(x)T^{2/3}/2}$. The integral 
$\tilde{I}_2$ admits the same estimate.
Therefore, we obtain from (\ref{lem11.1.5}) the relation
\begin{equation}\label{lem11.1.6}
\tilde{I} = \frac{1}{\sqrt{2\pi\alpha(x)}}
\Big(1-\frac {15}{2\alpha(x)^3\,T^2}+\frac{c\theta}{\alpha(x)^6\,T^4}\Big).
\end{equation}
Applying (\ref{lem11.1.6}) in (\ref{lem11.1.4}), we deduce for the 
half-axis $x\le(1+\sigma^2)^2\,T/16$, the formula 
\begin{equation}\label{lem11.1.7}
v_{\sigma}(x) =
\frac{1}{\sqrt{2\pi\alpha(x)}}\Big(1-\frac{15}{2\alpha(x)^3\,T^2} +
\frac{c\theta}{\alpha(x)^6\,T^4}\Big)\, e^{y(x)x}f_{\sigma}(iy(x)).
\end{equation}
We conclude immediately from 
(\ref{lem11.1.7}) that (\ref{lem11.1.2}) holds. To prove
(\ref{lem11.1.2'}), we use (\ref{lem11.1.4}) with 
$y = y_0 = -(1+\sigma^2)T/16$ and, noting that
\begin{equation}\notag
x + \frac{1+\sigma^2}2y_0\ge \frac{x}{2}\quad \text{for}\quad 
x\ge \frac{(1+\sigma^2)^2}{16}\,T,
\end{equation}
we easily deduce the desired estimate
\begin{equation}\notag
|v_{\sigma}(x)|\le e^{-(1+\sigma^2)Tx/32}\frac 1{2\pi}
\int_{-\infty}^\infty e^{-5(1+\sigma^2)^2t^2/16}\,dt \le 
e^{-(1+\sigma^2)Tx/32}.
\end{equation}

Finally, to prove (\ref{lem11.1.2''}), we apply the formula 
(\ref{lem11.1.7}). Using the explicit form of $y(x)$, write
\begin{equation}\label{lem11.1.8}
e^{y(x)x}f_{\sigma}(iy(x)) =
\exp\Big\{y(x)x+\frac{1+\sigma^2}{2}\,y^2(x)+\cfrac{y(x)^3}T\Big\}
=\exp\Big\{\frac {y(x)}3(2x+\frac{1+\sigma^2}2y(x))\Big\}
\end{equation}
for $x\le \frac{(1+\sigma^2)^2}{16}\,T$. Note that the function 
$y(x)/x$ is monotonically decreasing from zero to 
$-\frac{4}{3}\,(1+\sigma^2)^{-1}$ and is equal to
$= -\frac{8}{3}\,\big(-1+\sqrt{\frac{7}{4}}\,\big)\,(1+\sigma^2)^{-1}$
at the point $x=-\frac{(1+\sigma^2)^2}{16}\,T$. Using these properties 
in (\ref{lem11.1.8}), we conclude that in the interval
$|x|\le \frac{(1+\sigma^2)^2}{16}\,T$,
\begin{equation}\label{lem11.1.9}
e^{-2(5-\sqrt 7)|y(x)x|/9}\le e^{y(x)x}f_{\sigma}(iy(x))\le e^{-4|y(x)x|/9},
\end{equation}
where the right-hand side continues to hold for all 
$x\le\frac{(1+\sigma^2)^2}{16}\,T$. The inequalities in (\ref{lem11.1.2''}) 
follow immediately from (\ref{lem11.1.7}) and (\ref{lem11.1.9}).
\end{proof}

\vskip2mm
Now, introduce independent identically distributed random variables 
$U$ and $V$ with density 
\begin{equation}
p(x) = d_0 v_0(x)\,I_{(-\infty,T/16]}(x),\quad d_0=1/\int_{-\infty}^{T/16}v_0(u)\,du,
\end{equation}
where $I_A$ denotes the indicator function of a set $A$.
The density $p$ depends on $T$, but for simplicity we omit this parameter.
Note that, by Lemma~\ref{lem11.1.1}, $|1-d_0|\le e^{-cT^2}$.

Consider the regularized random variable $U_{\sigma}$ with density 
$p_{\sigma} = p*\varphi_{\sigma}$, which we represent in the form
\begin{equation}\notag
p_{\sigma}(x) = d_0v_{\sigma}(x)-w_{\sigma}(x),\quad\text{where}\quad
w_{\sigma}(x) = 
d_0\,((v_0I_{(T/16,\infty)})*\varphi_{\sigma})(x).
\end{equation}
The next lemma is elementary, and we omit its proof.

\begin{lemma}\label{lem11.1'}
We have
\begin{align}
|w_{\sigma}(x)| & \le
\varphi_{\sigma}(|x|+T/16)\,e^{-cT^2}, \qquad\qquad x\le 0,\notag\\
|w_{\sigma}(x)|&\le e^{-cT^2},\qquad\qquad\qquad \qquad \qquad \quad 0<x\le T/16,\notag\\
|w_{\sigma}(x)|&\le e^{-cTx},\qquad\qquad\qquad \qquad \qquad \quad x>T/16.\notag
\end{align}
\end{lemma}

\begin{lemma}\label{lem11.2}
For all sufficiently large $T>1$ and $0<\sigma\le 2$,
\begin{equation}\notag
D(U_{\sigma})= \frac{3}{(1+\sigma^2)^3\,T^2}+\frac {c\theta}{T^3}.
\end{equation}
\end{lemma}

\begin{proof} 
Put $\E U_{\sigma} = a_{\sigma}$ and $\Var (U_{\sigma})=b_{\sigma}^2$. 
By Lemma~\ref{lem11.1'},
$|a_{\sigma}|+|b_{\sigma}^2-1-\sigma^2|\le e^{-cT^2}$. Write
\begin{align}
&D(U_{\sigma})=\tilde{J}_1+\tilde{J}_2+\tilde{J}_3 =
d_0\int_{|x|\le c'T}v_{\sigma}(x)
\log\frac{p_{\sigma}(x)}{\varphi_{a_{\sigma},b_{\sigma}}(x)}\,dx\notag\\
&-\int_{|x|\le c'T}w_{\sigma}(x)
\log\frac{p_{\sigma}(x)}{\varphi_{a_{\sigma},b_{\sigma}}(x)}\,dx +
\int_{|x|>c'T}p_{\sigma}(x)
\log\frac{p_{\sigma}(x)}{\varphi_{a_{\sigma},b_{\sigma}}(x)}\,dx,\label{lem11.2.1}
\end{align}
where $c'>0$ is a sufficiently small absolute constant.
First we find lower and upper bounds of $\tilde{J}_1$, which are based 
on some additional information about $v_{\sigma}$.

Using a Taylor expansion for the function $\sqrt{1-u}$ about zero
in the interval $-\frac{3}{4} \le u \le \frac{3}{4}$, we easily obtain, 
for $|x|\le (1+\sigma^2)^2T/16$,
\bee
\frac{6y(x)}{(1+\sigma^2)T}
 & = &
-1+\sqrt{1-\frac {12x}{(1+\sigma^2)^2}T} \\
 & = &
-\frac{6x}{(1+\sigma^2)^2T} -\frac{18x^2}{(1+\sigma^2)^4T^2} -
\frac{108 x^3}{(1+\sigma^2)^6T^3}+\frac{c\theta x^4}{T^4},
\ene
which leads to the relation
\begin{equation}\label{lem11.2.2}
y(x)x + \frac{1+\sigma^2}{2}\, y(x)^2+\frac{y(x)^3}T=
-\frac{x^2}{2(1+\sigma^2)}-\frac{x^3}{(1+\sigma^2)^3T}
-\frac {9x^4}{2(1+\sigma^2)^5T^2}+
\frac{c\theta x^5}{T^3}.
\end{equation}
In addition, it is easy to verify that
\begin{equation}\label{lem11.2.3}
\alpha(x)=
(1+\sigma^2)\Big(1-\frac{6x}{(1+\sigma^2)^2T} -
\frac{18x^2}{(1+\sigma^2)^4T^2}+\frac{c\theta x^3}{T^3}\Big).
\end{equation}
Finally, using (\ref{lem11.2.2}) and (\ref{lem11.2.3}), 
we conclude from (\ref{lem11.1.7}) that $v_\sigma$ is representable as
\begin{align}
v_{\sigma}(x)&=g(x)\varphi_{\sqrt{1+\sigma^2}}(x)e^{h(x)}\notag \\
 & =
\Big(1+\frac{3x}{(1+\sigma^2)^2T}+\frac{15}2\,
\frac{3x^2-(1+\sigma^2)}{(1+\sigma^2)^4T^2}
+\frac{c\theta |x|(1+x^2)}{T^3}\Big) \notag\\
 & \qquad\qquad\varphi_{\sqrt{1+\sigma^2}}(x) \exp\Big\{-\frac{x^3}{(1+\sigma^2)^3T}
-\frac{9x^4}{2(1+\sigma^2)^5T^2}+ \frac{c\theta x^5}{T^3}\Big\}\label{lem11.2.4}
\end{align}
for $|x|\le (1+\sigma^2)^2T/16$.

Now, from (\ref{lem11.2.4}) and Lemma~\ref{lem11.1'}, we obtain 
a simple bound 
\begin{equation}\label{lem11.2.4'}
|w_{\sigma}(x)/v_{\sigma}(x)|\le 1/2\quad\text{ for}\quad |x|\le c'T.
\end{equation}
Therefore we have this relation, using again Lemma~\ref{lem11.1} and 
Lemma~\ref{lem11.1'},
\begin{align}
\tilde{J}_1
&=\int_{|x|\le 
c'T}v_{\sigma}(x)\log\frac{v_{\sigma}(x)}{\varphi_{a_{\sigma},b_{\sigma}}(x)}\,dx
+2\theta\int_{|x|\le c'T}|w_{\sigma}(x)|\,dx\notag\\
&=\int_{|x|\le c'T}
v_{\sigma}(x)\log\frac{v_{\sigma}(x)}{\varphi_{\sqrt{1+\sigma^2}}(x)}\,dx 
+ \theta e^{-cT^2}. \label{lem11.2.5}
\end{align}
Let us denote the integral on the right-hand side of (\ref{lem11.2.5}) 
by $\tilde{J}_{1,1}$. With the help of (\ref{lem11.2.4}) it is not 
difficult to derive the representation
\begin{align}
\tilde{J}_{1,1}=\int_{|x|\le c'T}&\varphi_{\sqrt{1+\sigma^2}}(x)e^{h(x)}
\Big(-\frac{x^3}{(1+\sigma^2)^3T}-\frac{15x^4}{2(1+\sigma^2)^5T^2}\notag\\
&+\frac{3x}{(1+\sigma^2)^2T}+\frac{54x^2-15(1+\sigma^2)}{2(1+\sigma^2)^4T^2}+
\frac{c\theta|x|(1+x^4)}{T^3}\Big)\,dx.
\label{lem11.2.6}
\end{align}
Since 
\begin{equation}\notag
|e^{h(x)}-1-h(x)|\le\frac 12h(x)^2e^{|h(x)|},
\end{equation}
and 
$\varphi_{\sqrt{1+\sigma^2}}(x)e^{h(x)}\le \sqrt{\varphi_{\sqrt{1+\sigma^2}}(x)}$ 
for  $|x|\le c'T$,
we easily deduce from (\ref{lem11.2.6}) that
\begin{align}
\tilde{J}_{1,1}=
\int_{|x|\le c'T}&\varphi_{\sqrt{1+\sigma^2}}(x)
\Big(\frac{3(1+\sigma^2)x-x^3}{(1+\sigma^2)^3T}
+\frac{54x^2-15(1+\sigma^2)}{2(1+\sigma^2)^4T^2}\notag\\
&-\frac{21(1+\sigma^2)x^4-2x^6}{2(1+\sigma^2)^6T^2}\Big)\,dx+
\frac{c\theta}{T^3}=\frac{3}{(1+\sigma^2)^3T^2}
+\frac{c\theta}{T^3}. \label{lem11.2.7}
\end{align}

It remains to estimate the integrals $\tilde{J}_2$ and $\tilde{J}_3$. 
By (\ref{lem11.2.4'}) and Lemma~\ref{lem11.1'},
\begin{equation}\label{lem11.2.8}
|\tilde{J}_2|\le\int_{|x|\le c'T}|w_{\sigma}(x)|
(-\log\varphi_{a_{\sigma},b_{\sigma}}(x)+
\log \frac 32+|\log v_{\sigma}(x)|)\,dx
\le \tilde{c}T^3e^{-cT^2}\le e^{-cT^2},
\end{equation}
while by Lemma~\ref{lem11.1} and Lemma~\ref{lem11.1'}, 
\begin{align}\label{lem11.2.9}
|\tilde{J}_3|&\le \int_{|x|>c'T}(|v_{\sigma}(x)|+|w_{\sigma}(x)|)
(\sqrt{2\pi}b_{\sigma} + \frac{x^2}{2b_{\sigma}^2}+|\log(|v_{\sigma}(x)
+|w_{\sigma}(x)|)\,dx\notag\\
&\le \tilde{c}\int_{|x|>c'T}(1+x^2)e^{-cT|x|}dx +
\int_{|x|>c'T}(|v_{\sigma}(x)|+|w_{\sigma}(x)|)^{1/2}\,dx\le e^{-cT^2}.
\end{align}
The assertion of the lemma follows from 
(\ref{lem11.2.7})--(\ref{lem11.2.9}).
\end{proof}

\vskip2mm
To complete the proof of Theorem~1.3, we need yet another lemma.

\begin{lemma}\label{lem11.3}
For all sufficiently large $T>1$ and $0<\sigma\le 2$, we have
$$
D(U_{\sigma}-V_{\sigma})\le e^{-cT^2}.
$$

\end{lemma}
\begin{proof}
Putting $\bar{p}_{\sigma}(x) = p_{\sigma}(-x)$, we have
\begin{align}\label{lem11.3.1}
D(U_{\sigma}-V_{\sigma})&=
\int_{-\infty}^\infty (p_{\sigma}*\bar{p}_{\sigma})(x)
\log\frac{(p_{\sigma}*\bar{p}_{\sigma})(x)}
{\varphi_{\sqrt{2(1+\sigma^2)}}(x)}\,dx\notag\\
 & +
\int_{-\infty}^\infty (p_{\sigma}*\bar{p}_{\sigma})(x)
\log\frac{\varphi_{\sqrt{2(1+\sigma^2)}}(x)}
{\varphi_{\sqrt{\Var(X_{\sigma}-Y_{\sigma})}}(x)}\,dx.
\end{align}
Note that $\bar{p}_{\sigma}(x)=d_0\bar{v}_{\sigma}(x)-\bar{w}_{\sigma}(x)$ 
with $\bar{v}_{\sigma}(x) = v_{\sigma}(-x)$, 
$\bar{w}_{\sigma}(x) = w_{\sigma}(-x)$, and 
\begin{equation}\label{lem11.3.2'}
p_{\sigma}*\bar{p}_{\sigma} =
d_0^2(v_{\sigma}*\bar{v}_{\sigma})(x)-d_0(v_{\sigma}*\bar{w}_{\sigma})(x)
-d_0(\bar{v}_{\sigma}*w_{\sigma})(x)+(w_{\sigma}*\bar{w}_{\sigma})(x).
\end{equation}
By the very definition of $v_{\sigma}$, $v_{\sigma}*\bar{v}_{\sigma}=
\varphi_{\sqrt{2(1+\sigma^2)}}$.
Since $|\Var(U_{\sigma}-V_{\sigma})-2(1+\sigma^2)|\le e^{-cT^2}$, 
using Lemma~\ref{lem11.1},
we note that the second integral on the right-hand side of (\ref{lem11.3.1}) 
does not exceed $e^{-cT^2}$. Using Lemma~\ref{lem11.1} and
Lemma~\ref{lem11.1'}, we get
\begin{align}
&|(v_{\sigma}*\bar{w}_{\sigma})(x)|+|(\bar{v}_{\sigma}*w_{\sigma})(x)|+
|w_{\sigma}*\bar{w}_{\sigma}(x)|
\le e^{-cT^2},\quad |x|\le \tilde{c}T,\label{lem11.3.2}\\
&|(v_{\sigma}*\bar{w}_{\sigma}(x)|+|(\bar{v}_{\sigma}*w_{\sigma})(x)|+
|(w_{\sigma}*\bar{w}_{\sigma})(x)|
\le e^{-cT|x|},\quad |x|> \tilde{c}T,\label{lem11.3.3}.
\end{align}
It follows from these estimates that
\begin{equation}
\frac{(p_{\sigma}*\bar{p}_{\sigma})(x)}{\varphi_{\sqrt{2(1+\sigma^2)}(x)}}=
1+c\theta e^{-cT^2}
\end{equation}
for $|x|\le c'T$. Hence, with the help of the Lemma~\ref{lem11.1} 
and Lemma~\ref{lem11.1'}, we may conclude that
\begin{equation}
\bigg|\int_{|x| \le c'T}
(p_{\sigma}*\bar{p}_{\sigma})(x)\log\frac{(p_{\sigma}*\bar{p}_{\sigma})(x)}
{\varphi_{\sqrt{2(1+\sigma^2)}(x)}}\,dx\bigg| \le e^{-cT^2}.
\end{equation}
A similar integral over the set $|x|>c'T$ can be estimated 
with the help of (\ref{lem11.3.2}) and (\ref{lem11.3.3}), and here we 
arrive at the same bound as well. Therefore, the assertion of the lemma follows 
from (\ref{lem11.3.1}).
\end{proof}

Introduce the random variables $X=(U-a_0)/b_0$ and $Y=(V-a_0)/b_0$.
Since $D(X_{\sigma})=D(U_{b_0\sigma})$ and $D(X_{\sigma}-Y_{\sigma})=D(U_{b_0\sigma}-
V_{b_0\sigma})$,
the statement of Theorem~1.3 for the entropic distance $D$ immediately 
follows from Lemma~\ref{lem11.2} and Lemma~\ref{lem11.3}.

As for the distance $J_{st}$, we need to prove corresponding analogs 
of Lemma~\ref{lem11.2} and Lemma~\ref{lem11.3} for $J_{st}(U_{\sigma})$ 
and $J_{st}(U_{\sigma}-V_{\sigma})$, respectively.
By the Stam inequality (\ref{l4}) and Lemma~\ref{lem11.2}, we see that 
\begin{equation}\label{lem11.3.4}
J_{st}(U_{\sigma}) \ge c(\sigma)\,T^{-2}\quad
\text{for sufficiently large} \ T>1,
\end{equation}
where $c(\sigma)$ denote positive constants depending on $\sigma$ only.
We estimate the quantity $J_{st}(U_{\sigma}-V_{\sigma})$, by using the formula
\begin{align}\label{lem11.3.5}
\frac{J_{st}(U_{\sigma}-V_{\sigma})}{2\Var(U_{\sigma})} =
-\int_{-\infty}^\infty (p_{\sigma}*\bar{p}_{\sigma})''(x)
\log\frac
{(p_{\sigma}*\bar{p}_{\sigma})(x)}
{\varphi_{\sqrt{\Var(U_{\sigma}-V_{\sigma})}}(x)}\,dx.
\end{align}
It is not difficult to conclude from (\ref{lem11.3.2'}), 
using our previous arguments, that
\begin{equation}\label{lem11.3.6}
(p_{\sigma}*\bar{p}_{\sigma})''(x)=d_0^2\varphi''_{\sqrt{2(1+\sigma^2)}}(x) + 
R_{\sigma}(x),
\end{equation}
where $|R_{\sigma}(x)|\le c(\sigma)e^{-cT^2}$ for $|x|\le \tilde{c} T$ 
and $|R_{\sigma}(x)|\le c(\sigma)e^{-cT|x|}$ for $|x|>\tilde{c} T$.
Applying (\ref{lem11.3.6}) in the formula (\ref{lem11.3.5}) and 
repeating the argument that we used in the proof of Lemma~\ref{lem11.3},
we obtain the desired result, namely
\begin{equation}\label{lem11.3.7}
J_{st}(U_{\sigma}-V_{\sigma})\le 
c(\sigma)\Var(X_{\sigma})e^{-cT^2}\quad\text{for sufficiently large}\ T>1.
\end{equation}
By Theorem 1.2, $J_{st}(U_{\sigma})\le 
-c(\sigma)/(\log J_{st}(U_{\sigma}-V_{\sigma}))$, which implies
$J_{st}(U_{\sigma})\to 0$ as $T\to\infty$. 
Since $J_{st}(X_{\sigma})=J_{st}(U_{b_0\sigma})$ and 
$J_{st}(X_{\sigma}-Y_{\sigma})=J_{st}(U_{b_0\sigma}-
V_{b_0\sigma})$,
the statement of Theorem~1.3 
for $J_{st}$ follows from (\ref{lem11.3.4}) and (\ref{lem11.3.7}).


\begin{thebibliography}{BH3}
\itemsep=-0pt
\small



\vskip2mm
\bibitem[B-N-T]{B-N-T}
Ball, K., Nayar, P., Tkocz, T. A reverse entropy power inequality for log-concave
random vectors. Preprint, arXiv:1509.05926, 2015.


\vskip2mm
\bibitem[B-C-G1]{B-C-G1}
Bobkov, S. G., Chistyakov, G. P., G\"otze, F.
         Entropic instability of Cramer's characterization of the normal 
         law. Selected works of Willem van Zwet, 231Â-242, Sel. Works 
         Probab. Stat., Springer, New York, 2012.

\vskip2mm
\bibitem[B-C-G2]{B-C-G2}
Bobkov, S. G., Chistyakov, G. P., G\"otze, F. Stability problems in
         Cramer-type characterization in case of i.i.d. summands.
         Theory Probab. Appl. 57 (2012), no. 4, 701Â-723.
         
\vskip2mm
\bibitem[B-C-G3]{B-C-G3}
Bobkov, S. G., Chistyakov, G. P., G\"otze, F. Regularized distributions
         and entropic stability of Cramer's characterization of the
         normal law. Extended version -- arXiv:1504.02961v1 [math.PR] 12 Apr 2015.

\vskip2mm
\bibitem[B-C-G4]{B-C-G4}
Bobkov, S. G., Chistyakov, G. P., G\"otze, F. Stability of Cramer's characterization of the
         normal law in information distances. Extended version --         
         
 \vskip2mm         
\bibitem[B-M1]{B-M1}			
Bobkov, S. G., Madiman, M. Reverse Brunn-Minkowski and reverse entropy power 
        inequalities for convex measures. J. Funct. Anal. 262 (2012), no. 7, 
				3309--3339.				
         
\vskip2mm         
\bibitem[B-M2]{B-M2}
Bobkov, S. G., Madiman, M. On the problem of reversibility of the entropy 
        power inequality. In: Limit theorems in probability, statistics and 
				number theory. Springer Proc. Math. Stat., 61--74, vol. 42, Springer, 
				Heidelberg, 2013.        
         
\vskip2mm
\bibitem[B-G-R-S]{B-G-R-S}
Bobkov, S. G., Gozlan, N., Roberto, C., Samson, P.-M. Bounds on the deficit 
        in the logarithmic Sobolev inequality. J. Funct. Anal. 267 (2014), 
        no. 11, 4110--4138.   

\vskip2mm         
\bibitem[B-M]{B-M}
Bobkov, S. G., Madiman, M. M. On the problem of reversibility of the entropy power inequality.
In Limit theorems in probability, statistics and number theory, vol. 42
of Springer Proc. Math. Stat., 61--74. Springer, Heidelberg, 2013.        
        
        
        
\vskip2mm
\bibitem[C]{C} 
Carlen, E. A. Superadditivity of Fisher's information and logarithmic 
        Sobolev inequalities. J. Funct. Anal. 101 (1991), no. 1, 194211. 

\vskip2mm  
\bibitem[C-Z]{C-Z}          
 Cover, T. M., Zhang, Z. On the maximum entropy of the sum of two depend random variables.
         IEEE Trans. Inform. Theory, 40, no 4, 1244--1246, 1994.        
        
        
\vskip2mm  
\bibitem[D-C-T]{D-C-T} 
Dembo, A., Cover, T. M., Thomas, J. A. Information-theoretic inequalities. 
         IEEE Trans. Inform. Theory, 37 (1991), no. 6, 1501--1518.

\vskip2mm
\bibitem[J]{J}
Johnson, O. Information theory and the central limit theorem. Imperial College 
        Press, London, 2004, xiv+209 pp. 


\vskip2mm
\bibitem[L-O]{L-O} 
Linnik, Yu. V., Ostrovskii, I. V. Decompositions of random variables and 
         random vectors. (Russian) Izd. Nauka, Moscow, 1972, 479 pp. 
         Translated from the Russian: Translations of Mathematical
         Monographs, Vol. 48. American Math. Society, Providence, 
         R. I., 1977, ix+380 pp. 

\vskip2mm
\bibitem[MK]{MK} 
McKean, H. P., Jr. Speed of approach to equilibrium for Kac's caricature of 
         a Maxwellian gas. Arch. Rational Mech. Anal. 21 (1966), 343--367.

\vskip2mm
\bibitem[S-G]{S-G}
Sansone, G., Gerretsen, J. Lectures on the theory of functions of 
         a complex variable I. (1960) P. Nordhoff-Groningen. The Netherlands.     

\vskip2mm
\bibitem[S]{S}
Stam, A. J. Some inequalities satisfied by the quantities of 
         information of Fisher and Shannon. Information and Control,    
         2:101--112, 1959.


\end{thebibliography}
\end{document}